\documentclass[a4paper,12pt]{amsart}

\usepackage{fullpage}
\usepackage{amsmath,amssymb,graphicx,psfrag}
\usepackage[T1]{fontenc}
\usepackage{amsmath,amssymb,ifthen}
\usepackage{color}
\usepackage{moreverb}

\newcommand{\N}{\mathbb{N}}
\newcommand{\R}{\mathbb{R}}
\newcommand{\HH}{\mathcal{H}}
\newcommand{\KK}{\mathcal{K}}
\newcommand{\MM}{\mathcal{M}}
\newcommand{\OO}{\mathcal{O}}
\newcommand{\RR}{\mathcal{R}}
\newcommand{\SSS}{\mathcal{S}}
\newcommand{\TT}{\mathcal{T}}

\newcommand{\UU}{\mathcal{U}}
\newcommand{\XX}{\mathcal{X}}

\newcommand{\dual}[3][]{#1\langle#2\,,\,#3#1\rangle}

\newcommand{\enorm}[2][]{#1|\!#1|\!#1|\,#2\,#1|\!#1|\!#1|}
\newcommand{\norm}[3][]{#1\|#2#1\|_{#3}}
\newcommand{\diam}{{\rm diam}}
\newcommand{\set}[3][\big]{#1\{#2\,:\,#3#1\}}
\newcommand{\refine}{{\rm refine}}

\newtheorem{lemma}{Lemma}
\newtheorem{theorem}[lemma]{Theorem}
\newtheorem{algorithm}[lemma]{Algorithm}
\newtheorem{remark}[lemma]{Remark}
\newtheorem{proposition}[lemma]{Proposition}
\newtheorem{corollary}[lemma]{Corollary}

\renewcommand{\subsection}[1]{\refstepcounter{subsection}\medskip{\bf\thesubsection.~#1.}}

\newenvironment{explain}{\begin{list}{$\bullet$}{%
\setlength{\labelsep}{2.3mm}%
\setlength{\labelwidth}{3mm}%
\setlength{\leftmargin}{5mm}%
}}{\end{list}}

\def\Cmesh{C_{\rm mesh}}
\def\Cson{C_{\rm son}}
\def\Cdopt{C_{1}}
\def\Copt{C_{\rm opt}}
\def\Cdopttwo{C_{2}}
\def\Cdoptthree{C_{3}}
\newcommand{\A}{\mathbb{A}}
\def\Cstab{C_{\rm stb}}

\def\plus{\bullet}%
\def\Cmark{C_{\rm mark}}
\def\qmesh{q_{\rm mesh}}
\newcommand{\T}{\mathbb{T}}

\def\eps{\varepsilon}

\newcommand{\oooldRevision}[1]{#1}
\newcommand{\ooldRevision}[1]{#1}
\newcommand{\oldRevision}[1]{#1}
\newcommand{\revision}[1]{#1}

\usepackage{fancyhdr}
\advance\footskip0.4cm
\advance\textheight-0.4cm
\calclayout
\newcommand*\patchAmsMathEnvironmentForLineno[1]{%
  \expandafter\let\csname old#1\expandafter\endcsname\csname #1\endcsname
  \expandafter\let\csname oldend#1\expandafter\endcsname\csname end#1\endcsname
  \renewenvironment{#1}%
     {\linenomath\csname old#1\endcsname}%
     {\csname oldend#1\endcsname\endlinenomath}}%
\newcommand*\patchBothAmsMathEnvironmentsForLineno[1]{%
  \patchAmsMathEnvironmentForLineno{#1}%
  \patchAmsMathEnvironmentForLineno{#1*}}%
\AtBeginDocument{%
\patchBothAmsMathEnvironmentsForLineno{equation}%
\patchBothAmsMathEnvironmentsForLineno{align}%
\patchBothAmsMathEnvironmentsForLineno{flalign}%
\patchBothAmsMathEnvironmentsForLineno{alignat}%
\patchBothAmsMathEnvironmentsForLineno{gather}%
\patchBothAmsMathEnvironmentsForLineno{multline}%
}
\usepackage[mathlines]{lineno}

\title{Adaptive FEM with Coarse Initial Mesh\\Guarantees Optimal Convergence Rates for \\Compactly Perturbed Elliptic Problems}

\author{Alex Bespalov}
\address{School of Mathematics, University of Birmingham, Edgbaston, Birmingham B15 2TT, UK}
\email{A.Bespalov@bham.ac.uk}

\author{Alexander Haberl}
\author{Dirk Praetorius}
\address{TU Wien, Institute for Analysis and Scientific Computing, Wiedner Hauptstra\ss{}e 8--10, 1040 Wien, Austria}
\email{Alexander.Haberl@asc.tuwien.ac.at}
\email{Dirk.Praetorius@asc.tuwien.ac.at\qquad\rm(corresponding author)}

\thanks{{\em Acknowledgements.} This work was initiated when DP visited the University of Birmingham to attend the \emph{Workshop on adaptive algorithms for computational PDEs}. The authors thank Gabriel Barrenechea for an informal and stimulating discussion on the topic of this work when they met at this workshop. \oldRevision{The authors also thank Christian Kreuzer for his comments on an earlier version of this paper that led to major improvements.} AH and DP acknowledge support through the research project \emph{Optimal adaptivity for BEM and FEM-BEM coupling}, funded by the Austria Science Fund (FWF) under grant P27005.}
\begin{document}

\begin{abstract}
We prove that for compactly perturbed elliptic problems, where the corresponding bilinear form satisfies a G\r{a}rding inequality, adaptive mesh-refinement is capable of overcoming the preasymptotic behavior and eventually leads to convergence with optimal algebraic rates. As an important consequence of our analysis, one does not have to deal with the {\sl a~priori} assumption that the underlying meshes are sufficiently fine. Hence, the overall conclusion of our results is that adaptivity has stabilizing effects and can overcome 
possibly pessimistic restrictions on the meshes.
In particular, our analysis covers adaptive mesh-refinement for the finite element discretization of the Helmholtz equation from where our interest originated.

\end{abstract}

\maketitle
\thispagestyle{fancy}

\section{Introduction}
\label{section:introduction}

%
\subsection{Adaptive mesh-refining algorithms}
{\sl A~posteriori} error estimation and related adaptive mesh-refinement is one fundamental column of finite element analysis. On the one hand, the {\sl a~posterori} error estimator allows to monitor whether the numerical solution is sufficiently accurate, even though the exact solution is unknown. On the other hand, its local contributions allow to adapt the underlying triangulation to resolve possible singularities most effectively. In recent years, the mathematical understanding of adaptive mesh-refinement has matured. It is now known that adaptive finite element methods (AFEM) of the type
\begin{align}\label{eq:algorithm}
\boxed{\tt~SOLVE~}\quad\Longrightarrow\quad
\boxed{\tt~ESTIMATE~}\quad\Longrightarrow\quad
\boxed{\tt~MARK~}\quad\Longrightarrow\quad
\boxed{\tt~REFINE~}
\end{align}
converge with optimal algebraic rate; see~\cite{doerfler,mns,bdd,stevenson07,ckns,ffp} to mention some milestones for conforming AFEM for linear second-order elliptic PDEs and~\cite{axioms} for some axiomatic approach. Essentially, only problems satisfying the Lax-Milgram theorem have been treated~\cite{doerfler,mns,bdd,stevenson07,ckns}. In a more general case of compactly perturbed elliptic problems, existing results have the limitation that the initial mesh has to be sufficiently fine~\cite{mn2005,cn,ffp}. 
On the other hand, numerical examples in the engineering literature
suggest that adaptive mesh-refinement performs well even if the initial mesh is coarse
(see, e.g., \cite{MR1393800,MR1643064,ihlenburg} in the case of the Helmholtz equation).
The purpose of this work is to bridge this gap at least for conforming elements.%

\subsection{Model problem}
\label{section:modelproblem}
Let $\Omega\subset\R^d$ be a polygonal resp.\ polyhedral Lipschitz domain, $d\ge2$. Let $\dual{f}{g}:=\int_\Omega fg\,dx$ denote the $L^2(\Omega)$ scalar product.
Suppose that $a(\cdot,\cdot)$ is a symmetric, continuous, and elliptic bilinear form on $\HH:=H^1_0(\Omega)$ and that $\KK:H^1_0(\Omega)\to L^2(\Omega)$ is a continuous linear operator. Given $f\in L^2(\Omega)$, we suppose that the variational formulation
\begin{align}\label{eq:weakform}
 b(u,v) := a(u,v) + \dual{\KK u}{v} = \dual{f}{v}
 \quad\text{for all }v\in \HH
 \end{align} 
admits a unique solution $u\in \HH$. Possible examples include the weak formulation of the Helmholtz equation 
\begin{align}\label{eq:helmholtz}
-\Delta u - \kappa^2u = f \text{ in }\Omega
\quad\text{subject to}\quad
u = 0 \text{ on }\partial\Omega,
\end{align}
where $\kappa^2\in\R$ is not an eigenvalue of $-\Delta$ and $\KK u = -\kappa^2u$, as well as more general 
diffusion problems with convection and reaction 
\begin{align}\label{eq:convection}
-{\rm div}(A\nabla u) + b\cdot\nabla u + cu = f \text{ in }\Omega
\quad\text{subject to}\quad
u = 0 \text{ on }\partial\Omega,
\end{align}
for given coefficients $c\in L^\infty(\Omega)$, $b\in L^\infty(\Omega)^d$, and $A\in L^\infty(\Omega)^{d\times d}$, where $A(x)\in\R^{d\times d}_{\rm sym}$ is symmetric and uniformly positive definite. We note that homogeneous Dirichlet conditions are only considered for the ease of presentation, while (inhomogeneous) mixed Dirichlet-Neumann-Robin boundary conditions can be included as in~\cite{dirichlet2d,dirichlet3d,axioms}.

We consider standard finite element spaces based on regular triangulations $\TT_\bullet$ of $\Omega$. For some fixed polynomial degree $p\ge1$, let
\begin{align*}
\SSS^p(\TT_\bullet) := \set{V_\bullet\in C(\Omega)}{\forall T\in\TT_\bullet\quad V_\bullet|_T\text{ is a polynomial } \text{of degree }\le p}
\end{align*} 
 be the usual finite element space of globally continuous piecewise polynomials and $\XX_\bullet := \SSS^p(\TT_\bullet)\cap H^1_0(\Omega)$ be the corresponding conforming subspace of $H^1_0(\Omega)$. Then, the discrete formulation reads as follows: Find $U_\bullet\in\XX_\bullet$ such that
\begin{align}\label{eq:discreteform}
 b(U_\bullet,V_\bullet)
 = \dual{f}{V_\bullet}
 \quad\text{for all }V_\bullet\in\XX_\bullet.
\end{align}
Let $h_\bullet\in L^\infty(\Omega)$ denote the local mesh-size function defined by $h_\bullet|_T := |T|^{1/d}$ for all $T\in\TT_\bullet$. Note that $h_\bullet|_T$ behaves like the diameter of the element $T\in\TT_\bullet$ on shape-regular meshes.
In general,~\eqref{eq:discreteform} may fail to allow for a (unique) solution $U_\bullet\in\XX_\bullet$. However, existence and uniqueness are guaranteed if $\TT_\bullet$ is sufficiently fine (see Corollary~\ref{cor:uniform}), e.g., 
$\norm{h_\bullet}{L^\infty(\Omega)}\le H\ll1$. Therefore, we employ one step of uniform refinement if~\eqref{eq:discreteform} does not allow for a unique solution $U_\bullet\in\XX_\bullet$.

\subsection{Contributions of present work}
Given an initial triangulation $\TT_0$, a typical adaptive algorithm~\eqref{eq:algorithm} generates a sequence of refined meshes $\TT_\ell$ with corresponding nested spaces $\XX_\ell\subseteq\XX_{\ell+1}\subset\HH$ for all $\ell\ge0$.
We stress that unlike prior works~\cite{mn2005,cn,ffp}, our adaptive algorithm (Algorithm~\ref{algorithm}) will not be given any information on whether the current mesh is sufficiently fine to allow for a unique solution.
In particular, we do not assume that the given initial mesh $\TT_0$ (and, in fact, any adaptive mesh $\TT_\ell$ generated by our algorithm) is sufficiently fine.
Nevertheless, we derive similar results as for uniformly elliptic problems (see, e.g., \cite{ckns,ffp,axioms} and the references therein), i.e., we prove linear convergence (Theorem~\ref{theorem:convergence}) with optimal algebraic convergence rates (Theorem~\ref{theorem:optimal}). More precisely, the framework and the main contributions of the present work can be summarized as follows:
\begin{itemize}
\item We consider a fixed mesh-refinement strategy that satisfies certain abstract assumptions (Section~\ref{section:mesh} and Section~\ref{section2:mesh}) which are met, e.g., for newest vertex bisection~\cite{stevenson08,kpp}.
\item We consider a fixed {\sl a~posteriori} error estimation strategy which satisfies the \emph{stability property on non-refined element domains}~\eqref{axiom:stability}, the \emph{reduction property on refined element domains}~\eqref{axiom:reduction}, and \oooldRevision{the \emph{reliabilty property}~\eqref{axiom:reliability:new} as well as} the \emph{discrete reliability property}~\eqref{axiom:reliability}. 
\item Under the above assumptions on the mesh-refinement and the error estimation strategy, we formulate our variant (Algorithm~\ref{algorithm}) of the adaptive loop~\eqref{eq:algorithm}, where marking is based on the D\"orfler marking  criterion introduced in~\cite{doerfler} with some adaptivity parameter $0<\theta\le1$. 
\item If the ``discrete'' limit space $\XX_\infty:=\overline{\bigcup_{\ell=0}^\infty\XX_\ell}$ satisfies an assumption~\oooldRevision{\eqref{axiom:infty}} which can be ensured by expanding the set of marked elements in the D\"orfler marking strategy (Section~\ref{section:marking}), we prove linear convergence (Theorem~\ref{theorem:convergence}) for any $0<\theta\le1$.
\item Starting from an index $L\in\N_0$, we prove that the C\'ea lemma is valid for the $a(\cdot,\cdot)$-induced energy norm and $\ell\ge L$, and the corresponding quasi-optimality constants converge to $1$ as $\ell\to\infty$ (Theorem~\ref{prop:cea}).
\item If additionally $0<\theta\ll1$ is sufficiently small \ooldRevision{and $\XX_\infty=\HH$ (which can be ensured by the expanded D\"orfler marking strategy mentioned above)}, we prove optimal algebraic convergence rates (Theorem~\ref{theorem:optimal}). While our presentation employs the estimator-based approximation classes from~\cite{axioms}, Section~\ref{section:classes} also discusses the relation to the approximation classes based on the total error from~\cite{ckns}.
\end{itemize}
We note that the entire analysis of this work applies to general situations, where $\HH$ is a separable Hilbert space over $\mathbb{K}\in\{\R,\mathbb{C}\}$, $\XX_\ell\subseteq\HH$ are conforming subspaces, and $\KK:\HH\to\HH^*$ is a compact operator; see Section~\ref{section:abstract}. 

%
\subsection{Outline}
Section~\ref{section:algorithm} provides the abstract framework of our analysis (Section~\ref{section:abstract}--\ref{section:mesh}) and gives a precise statement of the adaptive algorithm (Section~\ref{section:adaptive_algorithm}). Section~\ref{section:axioms} adapts~\cite{axioms} to the present setting and formulates certain properties of the error estimator. 
Section~\ref{section:convergence} proves convergence of the adaptive algorithm. Following~\cite{ffp}, we first prove plain convergence (Section~\ref{section:convergence:plain}) and then derive linear convergence (Section~\ref{section:convergence:linear}). Finally, we address the validity of the C\'ea lemma (Section~\ref{section:cea}).
Optimal algebraic convergence rates are the topic of Section~\ref{section:rates}, where we also discuss the involved approximation classes (Section~\ref{section:classes}). In the final Section~\ref{section:numerics}, we present numerical results for the 2D Helmholtz equation that underpin the developed theory.

\medskip

{\bf Notation.} We use $\lesssim$ to abbreviate $\le$ up to some (generic) multiplicative constant which is clear from the context. Moreover, $\simeq$ abbreviates that both estimates $\lesssim$ and $\gtrsim$ hold. Throughout, the mesh-dependence of (discrete) quantities is explicitly stated by use of appropriate indices, e.g., $U_\bullet$ is the discrete solution for the triangulation $\TT_\bullet$ and $\eta_\ell$ is the error estimator with respect to the triangulation $\TT_\ell$.

\section{Adaptive Algorithm}
\label{section:algorithm}
\subsection{Abstract setting}
\label{section:abstract}%
\def\K{\mathbb{K}}%
\def\C{\mathbb{C}}%
The model problem from Section~\ref{section:modelproblem} can be recast in the following abstract setting. Let $\HH$ be a separable Hilbert space over $\K\in\{\R,\C\}$. For each triangulation $\TT_\bullet$ with local mesh-size $h_\bullet\in L^\infty(\Omega)$, let $\XX_\bullet\subseteq\HH$ be a conforming finite-dimensional subspace. Suppose that $a(\cdot,\cdot)$ is a hermitian, continuous, and elliptic sesquilinear form on $\HH$, i.e., there exists some constant $\alpha>0$ such that
\begin{align}\label{eq:elliptic}
 \alpha\,\norm{v}\HH^2 \le a(v,v)
 \quad\text{for all }v\in\HH.
\end{align}
In particular, the $a(\cdot,\cdot)$-induced energy norm $\enorm{v}^2:=a(v,v)$ is an equivalent norm on $\HH$, i.e., $\enorm{v}\simeq\norm{v}{\HH}$ for all $v\in\HH$.
Let $\HH^*$ be the dual space of $\HH$, and let $\dual\cdot\cdot$ denote the corresponding duality pairing. Suppose that $\KK:\HH\to\HH^*$ is a compact linear operator and $f\in\HH^*$. In the remainder of this work, we consider the weak formulation~\eqref{eq:weakform} as well as its discretization~\eqref{eq:discreteform} within the above abstract framework.

The next proposition is an improved version of \cite[Theorem 4.2.9]{sauterschwab}. Even though the result appears to be well-known, we did not find the precise statement in the literature. We note that a similar result is proved in~\cite[Theorem~5.7.6]{brennerscott} under additional regularity assumptions for the dual problem. Instead, our proof below proceeds without considering the dual problem, and hence no additional regularity assumptions are needed. For these reasons and for the convenience of the reader, we include the following statement together with its proof.

\begin{proposition}
\label{prop:uniform}
Suppose well-posedness of~\eqref{eq:weakform}, i.e., 
\begin{align}\label{eq:wellposedness}
\forall w\in\HH\quad\big[w=0\quad\Longleftrightarrow\quad \big(\forall v\in\HH\quad b(w,v)=0\big)\big].
\end{align}
Suppose that \oldRevision{$(\XX_\ell)_{\ell\in\N_0}$ is a dense sequence of discrete subspaces $\XX_\ell\subset\HH$, i.e.,
\begin{align}\label{eq:density}
 \lim_{\ell\to\infty}\min_{V_\ell\in\XX_\ell}\revision{\norm{v-V_\ell}\HH}=0
 \quad\text{for all }v\in\HH.
\end{align}}%
Then, there exists
some index $\ell_\bullet\in\N_0$ such that for all discrete subspaces $\XX_\bullet\subset\HH$ with $\XX_\bullet \supseteq \XX_{\ell_\bullet}$, the following holds: \ooldRevision{\revision{There} exists $\gamma>0$ which depends only on \oooldRevision{$\XX_{\ell_\bullet}$, such} that the $\inf$-$\sup$ constant of $\XX_\bullet$ is uniformly bounded from below, i.e., 
\begin{align}\label{eq:infsup}
 \gamma_\bullet := \inf_{W_\bullet \in \XX_\bullet \setminus \{0\}} \sup_{V_\bullet \in \XX_\bullet \setminus \{0\}} \frac{|b(W_\bullet,V_\bullet)|}{\norm{W_\bullet}{\HH}\norm{V_\bullet}{\HH}} \ge \gamma > 0.
\end{align}}%
\oldRevision{In particular,} the discrete formulation~\eqref{eq:discreteform} admits a unique solution $U_\bullet\in\XX_\bullet$. \oldRevision{Moreover}, there holds uniform validity of the C\'ea lemma, i.e., there is a constant $C>0$ which depends \oldRevision{only on $b(\cdot,\cdot)$ and $\gamma$} but not on $\XX_\bullet$, such that
\begin{align}\label{eq:sauterschwab_cea}
 \norm{u-U_\bullet}{\HH} \le C\,\min_{V_\bullet \in \XX_\bullet}\norm{u-V_\bullet}{\HH}.
\end{align} 
If the spaces $\XX_\ell$ are nested, i.e., $\XX_\ell\subseteq\XX_{\ell+1}$ for all $\ell\in\N_0$, the latter guarantees convergence $\norm{u-U_\ell}{\HH}\to0$ as $\ell\to\infty$.
\end{proposition}

\begin{proof}
The bilinear form $b(\cdot,\cdot)$ induces \ooldRevision{the linear and continuous operator
\begin{align*}
\begin{array}{rcrcll}
B_\bullet: \XX_\bullet \rightarrow \XX_\bullet^\ast,&& \dual{B_\bullet W_\bullet}{V_\bullet} &:=& b(W_\bullet,V_\bullet) &\text{for all} \,\, V_\bullet,W_\bullet \in \XX_\bullet,
\end{array}
\end{align*}}%
where $\XX_\bullet$ is an arbitrary discrete subspace of $\HH$ with dual space $\XX_\bullet^*$.

\emph{Step 1: Discrete inf-sup condition.}\quad
Since $\XX_\bullet$ is finite dimensional and since we use the same discrete ansatz and test space, well-posedness of~\eqref{eq:discreteform} is equivalent to the discrete $\inf$-$\sup$ condition 
\begin{align}\label{eq:discrete_inf_sup_a}
 \gamma_\bullet = \inf_{W_\bullet \in \XX_\bullet \setminus \{0\}} \sup_{V_\bullet \in \XX_\bullet \setminus \{0\}} \frac{|b(W_\bullet,V_\bullet)|}{\norm{W_\bullet}{\HH}\norm{V_\bullet}{\HH}} = \inf_{W_\bullet \in \XX_\bullet \setminus \{0\}} \frac{\norm{B_\bullet W_\bullet}{\XX_\bullet^*}}{\norm{W_\bullet}\HH}
  > 0.
\end{align}
(Note that~\eqref{eq:discrete_inf_sup_a} implies that $B_\bullet$ is injective, and surjectivity follows from finite dimensionality of $\XX_\bullet$, i.e., $\dim\XX_\bullet = \dim\XX_\bullet^*<\infty$.)
Moreover, in this case there holds inequality~\eqref{eq:sauterschwab_cea} with 
\begin{align*}
 C:=1+\frac{M}{\gamma_\bullet},
 \quad\text{where}\quad
 M := \sup_{\substack{v\in\HH\backslash\{0\}\\w\in\HH\backslash\{0\}}}\frac{|b(w,v)|}{\norm{w}\HH\norm{v}\HH};
\end{align*}
see, e.g.,~\cite[Theorem~3.6, Lemma~3.7]{braess} or~\cite[Section~3]{demkowicz}.
Therefore, it is sufficient to prove the following assertion:
\begin{align}\label{eq:prop:claim}
 \exists\gamma>0 \, \exists\ell_\bullet\in\N_0 \, \forall \XX_\bullet\subset\HH\text{ with }\XX_\bullet\supseteq\XX_{\ell_\bullet}
 \quad \inf_{W_\bullet \in \XX_\bullet \setminus \{0\}} \frac{\norm{B_\bullet W_\bullet}{\XX_\bullet^*}}{\norm{W_\bullet}\HH} \ge \gamma.
\end{align}
We will prove~\eqref{eq:prop:claim} by contradiction.

\medskip

\emph{Step 2:}\quad Let us assume that~\eqref{eq:prop:claim} is wrong and hence
\begin{align}\label{eq:prop:assumption}
 \forall\gamma>0 \, \forall\ell_\bullet\in\N_0 \, \exists \XX_\bullet\subset\HH\text{ with }\XX_\bullet\supseteq\XX_{\ell_\bullet}
 \quad \inf_{W_\bullet \in \XX_\bullet \setminus \{0\}} \frac{\norm{B_\bullet W_\bullet}{\XX_\bullet^*}}{\norm{W_\bullet}\HH} < \gamma.
\end{align}
For each $\ell_\bullet = \ell\ge0$ and $\gamma=1/\ell$, we can thus find a discrete subspace $\XX_\bullet = \widehat\XX_\ell\subset\HH$ and an element $\widehat W_\ell \in \widehat\XX_\ell$ such that 
\begin{align*}
 \widehat\XX_\ell \supseteq \XX_\ell,
 \quad
 \norm{\widehat W_\ell}{\HH}=1,
 \quad\text{and}\quad
 \norm{\widehat B_\ell \widehat W_\ell}{\widehat\XX_\ell^*} < 1/\ell.
\end{align*}
Since the sequence $\widehat W_\ell$ is bounded and without loss of generality, we may assume weak convergence $\widehat W_\ell\rightharpoonup w\in\HH$ as $\ell\to\infty$.

\medskip

\emph{Step 3: There holds $w=0$.}\quad
Let $\widehat P_\ell:\HH \rightarrow \widehat{\XX}_\ell$ be the orthogonal projection onto $\widehat\XX_\ell$ and
$v\in\HH$. Then, weak convergence $\widehat W_\ell\rightharpoonup w$ and $b(\cdot,v)\in\HH^*$ prove $b(\widehat W_\ell,v) \rightarrow b(w,v)$ as $\ell \rightarrow \infty$.
Moreover, we employ $\norm{\widehat W_\ell}\HH=1$ and $\norm{\widehat P_\ell v}{\HH}\le\norm{v}{\HH}$
to estimate
\begin{align*}
 |b(\widehat W_\ell,v)| 
 &\le |b(\widehat W_\ell,\widehat P_\ell v)|+|b(\widehat W_\ell,v-\widehat P_\ell v)|
 \le \norm{\widehat B_\ell \widehat W_\ell}{\widehat \XX_\ell^*}\,\norm{v}{\HH} + \oldRevision{M\,\norm{v-\widehat P_\ell v}\HH}.
\end{align*}
\revision{Recall~\eqref{eq:density}} and $\XX_\ell\subseteq\widehat\XX_\ell\subset\HH$.
This implies
\begin{align*}
 \norm{v-\widehat P_\ell v}\HH = \min_{\widehat V_\ell\in\widehat\XX_\ell}\norm{v-\widehat V_\ell}\HH
 \le \min_{V_\ell\in\XX_\ell}\norm{v-V_\ell}\HH \xrightarrow{\ell\to\infty}0.
\end{align*}
Since $\norm{\widehat B_\ell \widehat W_\ell}{\widehat\XX_\ell^*}\le1/\ell$, we thus conclude that $|b(\widehat W_\ell,v)|\to0$ as $\ell\to\infty$. Altogether, $b(w,v)=0$ for all $v\in\HH$ and hence $w=0$.

\emph{Step 4: Assumption~\eqref{eq:prop:assumption} yields a contradiction so that~\eqref{eq:prop:claim} follows.}\quad Recall $\norm{\widehat W_\ell}\HH = 1$. Ellipticity of $a(\cdot,\cdot)$ and the definition of $b(\cdot,\cdot)$ yield
\begin{align*}
 \norm{\widehat W_\ell}\HH^2 \stackrel{\eqref{eq:elliptic}}
 \lesssim a(\widehat W_\ell,\widehat W_\ell)
 \le |b(\widehat W_\ell,\widehat W_\ell)| + |\dual{\KK\widehat W_\ell}{\widehat W_\ell}|
 \le \norm{\widehat B_\ell \widehat W_\ell}{\widehat\XX_\ell^*} + \norm{\KK\widehat W_\ell}{\HH^*}.
\end{align*}
Recall that compact operators turn weak convergence into strong convergence. Hence $\widehat W_\ell\rightharpoonup w=0$ in $\HH$ implies $\norm{\KK\widehat W_\ell}{\HH^*}\to0$ as $\ell\to\infty$. Together with $\norm{\widehat B_\ell \widehat W_\ell}{\widehat\XX_\ell^*}\le1/\ell$, we thus obtain the contradiction $1 = \norm{\widehat W_\ell}\HH\to0$ as $\ell\to\infty$.
\end{proof}%

\begin{remark}
To see that the model problem~\eqref{eq:weakform} fits into the abstract framework, recall that the Rellich theorem provides the compact inclusion $\HH := H^1_0(\Omega)\Subset L^2(\Omega)$. Therefore, the Schauder theorem (see, e.g.,~\cite[Theorem~4.19]{rudin}) implies the compact inclusion $L^2(\Omega)\Subset \HH^*$, where duality is understood with respect to the $L^2(\Omega)$ scalar product. Therefore, the continuous linear operator $\KK:H^1_0(\Omega) \to L^2(\Omega)$ turns out to be compact as an operator $\KK:\HH\to\HH^*$; see also the discussion in~\cite{ffp}.
\end{remark}

\begin{remark}
The work~\cite{ffp} considers problems, where the left-hand side of~\eqref{eq:weakform} is strongly elliptic on $\HH=H^1_0(\Omega)$, i.e., there exists $\widetilde\alpha>0$ such that
\begin{align}\label{eq:ffp}
 \widetilde\alpha\,\norm{v}{\HH}^2 \le {\rm Re}\,\big(a(v,v) + \dual{\KK v}{v}\big) 
 \quad\text{for all }v\in\HH.
\end{align}
Suppose that $a(w,w)>0$ for all $w\in\HH\backslash\{0\}$.
We note that~\eqref{eq:ffp} then already implies that $a(\cdot,\cdot)$ is elliptic in the sense of~\eqref{eq:elliptic}, so that the present work generalizes the analysis of~\cite{ffp}.
To see that~\eqref{eq:ffp} implies~\eqref{eq:elliptic}, we argue by contradiction, i.e., we assume the following: For all $\eps>0$, there is some $v\in\HH$ with $|a(v,v)|<\eps\,\norm{v}{\HH}^2$. Choosing $\eps=1/n$, we obtain a sequence $(v_n)$ in $\HH$ with $|a(v_n,v_n)|<\norm{v_n}{\HH}^2/n$. Define $w_n:=v_n/\norm{v_n}{\HH}$. Without loss of generality, we may thus suppose weak convergence $w_n\rightharpoonup w$ in $\HH$. Weakly lower semicontinuity proves $|a(w,w)|\le\liminf_{n\to\infty}|a(w_n,w_n)|=0$ and hence $w=0$. Therefore, compactness of $\KK$ yields $\norm{\KK w_n}{\HH^*}\to0$ as $n\to\infty$. Finally, ellipticity~\eqref{eq:ffp} gives $\widetilde\alpha = \widetilde\alpha\,\norm{w_n}\HH^2 \le {\rm Re}\,\big(a(w_n,w_n) + \dual{\KK w_n}{w_n}\big) < 1/n + \norm{\KK w_n}{\HH^*} \xrightarrow{n\to\infty}0$. This contradicts $\widetilde\alpha>0$, and we hence conclude that~\eqref{eq:ffp} implies~\eqref{eq:elliptic}.
\end{remark}%

\subsection{Mesh-refinement}
\label{section:mesh}%
From now on, suppose that $\TT_0$ is a given initial mesh.
Suppose that $\refine(\cdot)$ is a fixed mesh-refinement strategy (e.g., newest vertex bisection~\cite{stevenson08}) such that given a conforming triangulation $\TT_\star$ and $\MM_\star\subseteq\TT_\star$, the call $\TT_\plus=\refine(\TT_\star,\MM_\star)$ returns the coarsest conforming refinement $\TT_\plus$ of $\TT_\star$ such that all $T\in\MM_\star$ have been refined, i.e.,
\begin{explain}
\item $\TT_\plus$ is a conforming triangulation of $\Omega$;
\item for all $T\in\TT_\star$, it holds $T = \bigcup\set{T'\in\TT_\plus}{T'\subseteq T}$;
\item $\MM_\star\subseteq\TT_\star\backslash\TT_\plus$;
\item the number of elements $\#\TT_\plus$ is minimal amongst all other triangulations $\TT'$ which share the three foregoing properties.
\end{explain}
Furthermore, we write $\TT_\plus\in\refine(\TT_\star)$ if $\TT_\plus$ is obtained by a finite number of refinement steps, i.e., there exists $n\in\N_0$ as well as a finite sequence $\TT^{(0)},\dots,\TT^{(n)}$ of triangulations and corresponding sets $\MM^{(j)}\subseteq\TT^{(j)}$ such that
\begin{explain}
\item $\TT_\star = \TT^{(0)}$,
\item $\TT^{(j+1)} = \refine(\TT^{(j)},\MM^{(j)})$ for all $j=0,\dots,n-1$,
\item $\TT_\plus = \TT^{(n)}$.
\end{explain}
In particular, $\TT_\star\in\refine(\TT_\star)$. To abbreviate notation, we let $\T:=\refine(\TT_0)$ be the set of all possible triangulations which can be obtained from $\TT_0$.

We suppose that the refinement strategy yields a contraction of the local mesh-size function on refined elements, i.e., there exists $0<\qmesh<1$ such that $\TT_\plus\in\refine(\TT_\star)$ implies $h_\plus|_T \le\qmesh h_\star|_T$  for all $T\in\TT_\star\backslash\TT_\plus$. We note that $\qmesh={2^{-1/d}}$ for newest vertex bisection~\cite{stevenson08,ckns}. 

Finally, the following assumptions are clearly satisfied for the model problem from Section~\ref{section:modelproblem}, but have to be supposed explicitly in the abstract framework of Section~\ref{section:abstract}.
First, each triangulation $\TT_\star$ corresponds to a discrete subspace $\XX_\star\subset\HH$, and $\TT_\bullet\in\refine(\TT_\star)$ implies nestedness $\XX_\star\subseteq\XX_\bullet$. 
Second, iterated uniform mesh-refinement leads to a dense subspace of $\HH$, i.e., for $\widehat\TT_0:=\TT_0$ and the inductively defined sequence $\widehat\TT_{\ell+1}:=\refine(\widehat\TT_\ell,\widehat\MM_\ell)$ with $\widehat\MM_\ell\subseteq\widehat\TT_\ell$ for all $\ell\in\N_0$, it holds the following: If $\#\set{\ell\in\N_0}{\widehat\MM_\ell = \widehat\TT_\ell}=\infty$ (i.e., there are infinitly many steps that perform uniform refinement), then $\HH = \overline{\bigcup_{\ell=0}^\infty\widehat\XX_\ell}$. 

Under these assumptions, the following statement holds as an immediate consequence of Proposition~\ref{prop:uniform}.

\begin{corollary}\label{cor:uniform}
Let $\widehat\TT_0:=\TT_0$ and $\widehat\TT_{\ell+1}:=\refine(\widehat\TT_\ell,\widehat\MM_\ell)$ with $\widehat\MM_\ell\subseteq\widehat\TT_\ell$ for all $\ell\in\N_0$. Suppose that $\#\set{\ell\in\N_0}{\widehat\MM_\ell = \widehat\TT_\ell}=\infty$. 
Then, there exists $m\in\N_0$ \oldRevision{and $\gamma>0$}
such that for all discrete spaces $\XX_\bullet\subset\HH$ with $\XX_\bullet\supseteq\widehat\XX_m$ \oldRevision{the related $\inf$-$\sup$ constant~\eqref{eq:infsup} satisfies $\gamma_\bullet\ge\gamma>0$. 
In particular, $\XX_\bullet$ admits}
a unique solution $U_\bullet\in\XX_\bullet$ of~\eqref{eq:discreteform} which is quasi-optimal in the sense of inequality~\eqref{eq:sauterschwab_cea}. \oldRevision{Moreover,} the Galerkin solutions $\widehat U_\ell\in\widehat\XX_\ell$, for $\ell\ge m$, yield convergence $\lim\limits_{\ell\to\infty}\norm{u-\widehat U_\ell}\HH = 0$.\qed
\end{corollary}%

\def\Cstab{C_{\rm stb}}
\def\Cred{C_{\rm red}}
\def\Crel{C_{\rm rel}}
\def\Corth{C_{\rm orth}}
\def\qred{q_{\rm red}}
\def\eps{\varepsilon}
\def\diam{{\rm diam}}
\subsection{\textsl{A posteriori} error estimation}
\label{section:axioms}%
Let $\TT_\star\in\T=\refine(\TT_0)$.
We suppose that given the solution $U_\star\in\XX_\star$ of~\eqref{eq:discreteform} and $T\in\TT_\star$, we can compute some local refinement indicators $\eta_\star(T)\ge0$ as well as the related {\sl a~posteriori} error estimator
\begin{align}
 \eta_\star := \eta_\star(\TT_\star),
 \quad\text{where }
 \eta_\star(\UU_\star) := \Big(\sum_{T\in\UU_\star}\eta_\star(T)^2\Big)^{1/2}
 \text{ for all }\UU_\star\subseteq\TT_\star.
\end{align}
To prove convergence with optimal algebraic rates for Algorithm~\ref{algorithm}, we rely on the following \emph{axioms of adaptivity} which are slightly generalized when compared to those of~\cite{axioms}, since we always have to suppose solvability of the related discrete problem~\eqref{eq:discreteform}. 
\begin{enumerate}
\renewcommand{\theenumi}{A\arabic{enumi}}
\item\label{axiom:stability}\textbf{Stability on non-refined element domains:}
There exists $\Cstab>0$ such that for all $\TT_\star\in \T$ and all $\TT_\plus\in\refine(\TT_\star)$, the following implication holds: Provided that the discrete solutions $U_\star\in\XX_\star$ and $U_\plus\in\XX_\plus$ exist, it holds $\big|\eta_\plus(\TT_\plus\cap\TT_\star)-\eta_\star(\TT_\plus\cap\TT_\star)\big| \le \Cstab\,\norm{U_\plus-U_\star}{\HH}$.

\item\label{axiom:reduction}\textbf{Reduction on refined element domains:}
There exist $\Cred>0$ and $0<\qred<1$ such that for all $\TT_\star\in\T$ and all $\TT_\plus\in\refine(\TT_\star)$, the following implication holds: Provided that the discrete solutions $U_\star\in\XX_\star$ and $U_\plus\in\XX_\plus$ exist, it holds $\eta_\plus(\TT_\plus\backslash\TT_\star)^2 \le \qred\,\eta_\star(\TT_\star\backslash\TT_\plus)^2 + \Cred^2\,\norm{U_\plus-U_\star}{\HH}^2$.

\oooldRevision{\item\label{axiom:reliability:new}\textbf{Reliability:}
There exists $\Crel'>0$ such that for all $\TT_\star\in\T$, the following implication holds: Provided that the discrete solution $U_\star\in\XX_\star$ exists, it holds $\norm{u-U_\star}{\HH} \le \Crel'\,\eta_\star$.}

\item\label{axiom:reliability}\textbf{Discrete reliability:}
There exists $\Crel>0$ such that for all $\TT_\star\in\T$ and all $\TT_\plus\in\refine(\TT_\star)$, there exists a set $\RR_{\star,\plus}\subseteq\TT_\star$ such that the following implication holds: Provided that the discrete solutions $U_\star\in\XX_\star$ and $U_\plus\in\XX_\plus$ exist, it holds $\norm{U_\plus-U_\star}{\HH} \le \Crel\,\oooldRevision{\gamma_\bullet^{-1}}\,\eta_\star(\RR_{\star,\plus})$ as well as $\TT_\star\backslash\TT_\plus \subseteq \RR_{\star,\plus}$ with $\#\RR_{\star,\plus} \le \Crel\,\#(\TT_\star\backslash\TT_\plus)$, \oooldRevision{where $\gamma_\bullet>0$ is the $\inf$-$\sup$ constant~\eqref{eq:infsup} associated with $\XX_\bullet$.}

\end{enumerate}
\begin{remark}
For a general diffusion problem~\eqref{eq:convection} with piecewise Lipschitz diffusion coefficient $A\in W^{1,\infty}(T_0)$ for all $T_0\in\TT_0$ and $\XX_\star:=\SSS^p(\TT_\star)\cap H^1_0(\Omega)$, the local contributions of the usual residual error estimator read,
for all $T\in\TT_\star$,
\begin{align}\label{eq:indicators}
 \eta_\star(T)^2 
 = h_T^2\,\norm{f+{\rm div}(A\nabla U_\star)-b\cdot\nabla U_\star-cU_\star}{L^2(T)}^2
 + h_T\,\norm{[(A\nabla U_\star)\cdot n]}{L^2(\partial T\cap\Omega)}^2,
\end{align}
where $[(\cdot)\cdot n]$ denotes the normal jump over interior facets and $h_T:=|T|^{1/d} \simeq \diam(T)$.
For the Helmholtz problem~\eqref{eq:helmholtz}, these local contributions simplify to 
\begin{align}
 \eta_\star(T)^2 
 = h_T^2\,\norm{f+\Delta U_\star+\kappa^2\,U_\star}{L^2(T)}^2
 + h_T\,\norm{[\partial_n U_\star]}{L^2(\partial T\cap\Omega)}^2.
\end{align}
We note that in either case \eqref{axiom:stability}--\eqref{axiom:reliability} are already known with $\RR_{\star,\plus} = \TT_\star\backslash\TT_\plus$, and the corresponding constants depend only on uniform shape regularity of the triangulations $\TT_\star\in\T$ and the well-posedness of the continuous problem~\eqref{eq:weakform};
see~\cite{ckns,cn,ffp}.
The error estimator can be extended to mixed Dirichlet-Neumann-Robin boundary conditions, where inhomogeneous Dirichlet conditions are discretized by nodal interpolation for $d=2$ and $p=1$, see~\cite{dirichlet2d}, or by Scott-Zhang interpolation for $d\ge2$ and $p\ge1$, see~\cite{axioms}. In any case \eqref{axiom:stability}--\eqref{axiom:reliability} remain valid~\cite{dirichlet2d,axioms}, but $\RR_{\star,\bullet}$ consists of a fixed patch of $\TT_\star\backslash\TT_\bullet$~\cite{dirichlet3d,axioms}.
\end{remark}

\oldRevision{\begin{remark}
In usual situations, reliability~\eqref{axiom:reliability:new} already follows from discrete reliability~\eqref{axiom:reliability}; see \ooldRevision{Lemma~\ref{prop:new}}~{\rm(d)} below.
\end{remark}}%

\subsection{Adaptive algorithm}
\label{section:adaptive_algorithm}
Based on the {\sl a~posteriori} error estimator from the previous section,
we consider the following adaptive algorithm.

\begin{algorithm}\label{algorithm}
\textsc{Input:} Parameters $0<\theta\le1$ and $\Cmark\ge1$ as well as initial triangulation $\TT_0$ with $U_{-1}:=0\in\XX_0$  and $\eta_{-1}:=1$.\\
\textsc{Adaptive loop:} For all $\ell=0,1,2,\dots$, iterate the following steps~{\rm(i)--(v)}:
\begin{itemize}
\item[\rm(i)] If~\eqref{eq:discreteform} does not admit a unique solution in $\XX_\ell$, define $U_\ell:=U_{\ell-1}\in\XX_\ell$ and $\eta_\ell:=\eta_{\ell-1}$, let $\TT_{\ell+1}:=\refine(\TT_{\ell},\TT_{\ell})$ be the uniform refinement of $\TT_\ell$, increase $\ell$ by $1$, and continue with step~{\rm(i)}.
\item[\rm(ii)] Compute the unique solution $U_\ell\in\XX_\ell$ to~\eqref{eq:discreteform}.
\item[\rm(iii)] Compute the corresponding indicators $\eta_\ell(T)$ for all $T\in\TT_\ell$.
\item[\oldRevision{\rm(iv)}] \oldRevision{Determine} a set $\MM_\ell\subseteq\TT_\ell$ of up to the multiplicative constant $\Cmark$ minimal cardinality such that $\theta\eta_\ell^2\le \eta_\ell(\MM_\ell)^2$.
\item[\oldRevision{\rm(v)}] Compute $\TT_{\ell+1}:=\refine(\TT_\ell,\MM_\ell)$, increase $\ell$ by 1, and continue with step~{\rm(i)}.
\end{itemize}
\textsc{Output:} Sequences of successively refined triangulations $\TT_\ell$, discrete solutions $U_\ell$, and corresponding estimators $\eta_\ell$.%
\end{algorithm}

\begin{remark}
\begin{explain}
\item Apart from step~{\rm(i)}, Algorithm~\ref{algorithm} is the usual adaptive loop based on the D\"orfler marking strategy~\cite{doerfler} in \oooldRevision{\oldRevision{step~{\rm(iv)}}} as used, e.g., in~\cite{ckns,ffp,axioms}.
\item While $\Cmark=1$ requires to sort the indicators and hence leads to log-linear effort, Stevenson~\cite{stevenson07} showed that $\Cmark=2$ allows to determine $\MM_\ell$ in linear complexity.
\end{explain}
\end{remark}

To abbreviate notation, we define $\T := \refine(\TT_0)$ as the set of all possible refinements of the given initial mesh $\TT_0$ in Algorithm \ref{algorithm}. The following lemma exploits the validity of Proposition~\ref{prop:uniform} for uniform mesh-refinement (Corollary~\ref{cor:uniform}).

\begin{lemma}\label{lemma:ell1}
Let $(U_\ell)_{\ell\in\N_0}$ be the sequence of discrete solutions generated by Algorithm~\ref{algorithm}.
Then, there exists a minimal index $\ell_0\in\N_0$ such that~\eqref{eq:discreteform} does not admit a unique solution in $\XX_\ell$ for $0\le \ell<\ell_0$, but admits a unique solution $U_{\ell_0}\in\XX_{\ell_0}$. In particular, the corresponding mesh $\TT_{\ell_0}$ is the $\ell_0$-times uniform refinement of $\TT_0$.
Furthermore, there exists $\ell_1\in\N_0$ such that~\eqref{eq:discreteform} admits a unique solution $U_\ell\in\XX_\ell$ for all steps $\ell\ge\ell_1$ of Algorithm~\ref{algorithm}.
\end{lemma}%

\begin{proof}
Thanks to Corollary~\ref{cor:uniform}, the uniform refinement in step~(i) of Algorithm~\ref{algorithm} will only be performed finitely many times. This concludes the proof.
\end{proof}

\oldRevision{To prove convergence of Algorithm~\ref{algorithm},} we need an additional assumption (see~\oooldRevision{\eqref{axiom:infty}} below) which goes beyond the axioms in~\cite{axioms}. To that end, let us define the ``discrete'' limit space $\XX_\infty := \overline{\bigcup_{\ell=0}^\infty\XX_\ell}$. Because of nestedness $\XX_\ell \subseteq \XX_{\ell+1}$ for all $\ell\ge0$, $\XX_\infty$ is a closed subspace of $\HH$ and hence a Hilbert space. 
\begin{enumerate}
\setcounter{enumi}{4}
\renewcommand{\theenumi}{A\arabic{enumi}}
\item\label{axiom:infty}\textbf{Definiteness of $\boldsymbol{b(\cdot,\cdot)}$ on $\boldsymbol{\XX_\infty}$:}
For all $w\in\XX_\infty$, the following implication holds: If $b(w,v)=0$ for all $v\in\XX_\infty$, then $w=0$.
\end{enumerate}%
Clearly,~\oooldRevision{\eqref{axiom:infty}} is satisfied if $b(\cdot,\cdot)$ is ellipitic~\eqref{eq:ffp}. Moreover, note
that well-posedness~\eqref{eq:wellposedness} of~\eqref{eq:weakform} implies that~\oooldRevision{\eqref{axiom:infty}} is satisfied, if $\XX_\infty=\HH$. \oldRevision{In many generic situations, the identity $\XX_\infty = \HH$ is automatically satisfied, but it may also be enforced explicitly by expanding the set of marked elements in the D\"orfler marking criterion in step~(iv) of Algorithm~\ref{algorithm}; see Section~\ref{section:marking} below.}

\oldRevision{The following technical lemma exploits the validity of~\eqref{axiom:infty}.}

\def\Cmon{C_{\rm mon}}
\oooldRevision{\begin{lemma}\label{prop:new}
Suppose~\eqref{axiom:stability}, \eqref{axiom:reduction}, \eqref{axiom:reliability}, and~\eqref{axiom:infty}. \ooldRevision{Employ the notation of Algorithm~\ref{algorithm} for} $0<\theta\le1$.
Then, there exists $\ell_2\in\N_0$ and \ooldRevision{$\gamma>0$ such that for all $\TT_\bullet\in\refine(\TT_{\ell_2})$ with $\XX_\bullet\subseteq\XX_\infty$, the following \oldRevision{assertion~{\rm(a)}} holds}:
\begin{itemize}
\item[\oldRevision{\rm(a)}] \oldRevision{The corresponding $\inf$-$\sup$ constant~\eqref{eq:infsup} is bounded from below by $\gamma_\bullet\ge\gamma>0$. In particular, there exists a unique Galerkin solution $U_\bullet\in\XX_\bullet$ to~\eqref{eq:discreteform} which \ooldRevision{is} quasi-optimal in the sense of inequality~\eqref{eq:sauterschwab_cea}.}
\end{itemize}
\ooldRevision{Moreover, let $\TT_\star\in\T$ and $\TT_\bullet\in\refine(\TT_\star)\cap\refine(\TT_{\ell_2})$ and suppose that the Galerkin solution $U_\star\in\XX_\star$ exists. Then, there hold the following \revision{assertions~{\rm(b)--(c)}} with some additional constant $\Cmon>0$ which depends only on $\Cstab$, $\Cred$, $\Crel$, and $\gamma$:}
\begin{itemize}
\item[\oldRevision{\rm(b)}] \emph{uniform discrete reliability}, i.e., $\norm{U_\bullet-U_\star}\HH \le \Crel\,\gamma^{-1}\,\eta_\star(\RR_{\star,\bullet})$.
\item[\oldRevision{\rm(c)}] \emph{quasi-monotonicity of error estimator}, i.e., $\eta_\bullet\le\Cmon\,\eta_\star$.
\end{itemize}
If in addition $\XX_\infty=\HH$, \ooldRevision{then} the following \oldRevision{assertion~{\rm(d)}} holds:
\begin{itemize}
\item[\oldRevision{\rm(d)}] discrete reliability~\eqref{axiom:reliability} implies reliability~\eqref{axiom:reliability:new}, i.e., \ooldRevision{$\norm{u-U_\bullet}\HH \le \Crel\,\gamma^{-1}\,\eta_\bullet$.}
\end{itemize}
\end{lemma}}%

\oooldRevision{\begin{proof}
Employ Proposition~\ref{prop:uniform} with $\HH$ replaced by $\XX_\infty$. This proves~\oldRevision{(a)} and provides $\ell_2\in\N_0$ and $\gamma>0$ such that the $\inf$-$\sup$ constant~\eqref{eq:infsup} \oldRevision{for} all discrete subspaces \ooldRevision{$\XX_\bullet\subseteq\XX_\infty$ with $\XX_\bullet\supseteq\XX_{\ell_2}$} is uniformly bounded from below by \ooldRevision{$\gamma_\bullet\ge\gamma>0$}. Together with~\eqref{axiom:reliability}, this also proves~\oldRevision{(b)}. 
Moreover,~\oldRevision{(b)} allows to apply {\cite[Lemma~3.5]{axioms}} to obtain the quasi-monotonicity~\oldRevision{(c)}. Finally,~\oldRevision{(d)} follows from~\oldRevision{(b)} and \cite[Lemma~3.4]{axioms},
since uniform refinement yields convergence (see Corollary~\ref{cor:uniform}).
\end{proof}}%

\section{Convergence}
\label{section:convergence}%

\subsection{Convergence of adaptive algorithm}
\label{section:convergence:plain}%
This section proves that Algorithm~\ref{algorithm} guarantees convergence $\norm{u-U_\ell}{\HH}\to0$ as $\ell\to\infty$.

\begin{proposition}\label{lemma1:convergence}
Suppose~\oldRevision{\eqref{axiom:stability}--\eqref{axiom:infty}} and $0<\theta\le1$.  Employ the notation of Algorithm~\ref{algorithm}.  Then, the ``discrete'' limit space $\XX_\infty = \overline{\bigcup_{\ell=0}^\infty\XX_\ell}$ contains the exact solution to problem~\eqref{eq:weakform}, i.e., $u\in\XX_\infty$. Moreover, $\lim_{\ell\to\infty}\norm{u-U_\ell}{\HH} = 0 = \lim_{\ell\to\infty}\eta_\ell$.
\end{proposition}

\oooldRevision{The proof of Proposition~\ref{lemma1:convergence} relies on the following estimator reduction which (in a weaker form) \oldRevision{is first found in}~\cite{ckns}.}

\def\qest{q_{\rm est}}
\def\Cest{C_{\rm est}}
\begin{lemma}[generalized estimator reduction~{\cite[Lemma~9]{goafem}}]
\label{lemma:reduction}
Stability~\eqref{axiom:stability} and reduction~\eqref{axiom:reduction} together with the D\"orfler marking strategy from \oooldRevision{\oldRevision{step~{\rm(iv)}}} of Algorithm~\ref{algorithm} imply the following perturbed contraction: For each $\ell\in\N_0$ and all $\TT_\star\in\refine(\TT_{\ell+1})$ such that the discrete solutions $U_\ell\in\XX_\ell$ and $U_\star\in\XX_\star$ exist, it holds
$\eta_{\star}^2 \le \qest\,\eta_\ell^2 + \Cest\,\norm{U_\star-U_\ell}{\HH}^2$.
The constants $\Cest>0$ and $0<\qest<1$ depend only on~\eqref{axiom:stability}--\eqref{axiom:reduction} and on $0<\theta\le1$.\qed
\end{lemma}

\begin{proof}[Proof of Proposition~\ref{lemma1:convergence}]
\oldRevision{Let $\ell_2\in\N_0$ be the index defined in Lemma~\ref{prop:new}.
Without loss of generality, we may assume $\ell_2=0$ throughout the proof.} In order to prove that $\eta_\ell\to0$ as $\ell\to\infty$, we show that each subsequence $(\eta_{\ell_k})_{k\in\N_0}$ of the estimator sequence $(\eta_\ell)_{\ell\in\N_0}$ contains a further subsequence $(\eta_{\ell_{k_j}})_{j\in\N_0}$ with $\eta_{\ell_{k_j}}\to0$ as $j\to\infty$. According to basic calculus, this is in fact equivalent to $\eta_\ell\to0$ as $\ell\to\infty$.

\oldRevision{
\textit{Step 1: Boundedness of estimator sequence.}\quad
We apply Lemma~\ref{prop:new} with $\ell_2=0$. The quasi-monotonicity of the error estimator proves 
$\eta_\ell \le \Cmon\,\eta_0
 \text{ for all }\ell\in\N_0.$}

\textit{Step 2: Weak convergence of discrete solutions (subsequence).}\,\,
Recall the $a(\cdot,\cdot)$-induced energy norm $\enorm\cdot$.
\oooldRevision{From \oooldRevision{reliability~\eqref{axiom:reliability:new}} \oldRevision{and step~1}, we infer that
\begin{align*}\enorm{U_\ell}
\le \enorm{u} + \enorm{u-U_\ell}
\lesssim \enorm{u} + \sup_{\ell\in\N_0}\eta_\ell < \infty,
\end{align*}}%
i.e., the sequence of discrete solutions is uniformly bounded in $\HH$. Let $(\eta_{\ell_k})_{k\in\N_0}$ be an arbitrary subsequence of $(\eta_\ell)_{\ell\in\N_0}$ with corresponding discrete solutions $U_{\ell_k}$. Since $U_{\ell_k} \in \XX_{\ell_k} \subseteq \XX_\infty$, there exists a subsequence $(U_{\ell_{k_j}})_{j\in\N_0}$ of $(U_{\ell_k})_{k\in\N_0}$ and some limit $w\in\HH$ such that $U_{\ell_{k_j}}\rightharpoonup w$ weakly in $\HH$ as $j\to\infty$. According to Mazur's lemma (see, e.g.,~\cite[Theorem~3.12]{rudin}), convexity and closedness imply that $\XX_\infty$ is also closed with respect to the weak topology and hence $w\in\XX_\infty$.
Let $v\in\XX_\infty$. 
Let $P_\ell:\HH\to\XX_\ell$ be the orthogonal projection with respect to $\enorm\cdot$, i.e.,\begin{align*}
 \enorm{v-P_\ell v}
 = \min_{V_\ell\in\XX_\ell}\enorm{v-V_\ell}
 \quad\text{for all }v\in \HH.
\end{align*}
By definition of $\XX_\infty$, this also implies strong convergence $\enorm{v-P_\ell v}\to0$ as $\ell\to\infty$.
Recall that the product of a weakly convergent sequence and a strongly convergent sequence leads to convergence of the scalar product. Moreover, compact operators turn weak convergence into strong convergence, i.e., $\KK U_{\ell_{k_j}}\to \KK w$ strongly in $\HH^*$ as $j\to\infty$.
With these two observations, we derive
\begin{eqnarray*}
 0 \stackrel{\eqref{eq:discreteform}}= \dual{f}{P_{\ell_{k_j}}v} - a(U_{\ell_{k_j}},P_{\ell_{k_j}}v) - \dual{\KK U_{\ell_{k_j}}}{P_{\ell_{k_j}}v}
 \xrightarrow{j\to\infty}
 \dual{f}{v} - a(w,v) - \dual{\KK w}{v}.
\end{eqnarray*}
This proves that the weak limit $w\in\XX_\infty$ solves the Galerkin formulation
\begin{align}\label{eq:limitform}
 a(w,v) + \dual{\KK w}{v} = \dual{f}{v}
 \quad\text{for all }v\in\XX_\infty.
\end{align}

\textit{Step 3: Strong convergence of discrete solutions (subsequence).}\,
Note that $\enorm{w-U_{\ell_{k_j}}}^2 = \enorm{w}^2 - 2\,{\rm Re}\,a(w,U_{\ell_{k_j}})+ \enorm{U_{\ell_{k_j}}}^2$. Therefore, strong convergence $\enorm{w-U_{\ell_{k_j}}}\to0$ is equivalent to weak convergence $U_{\ell_{k_j}}\rightharpoonup w$ plus convergence of the norm $\enorm{U_{\ell_{k_j}}}\to\enorm{w}$. It thus only remains to prove the latter. 
With the previous observations, it holds
\begin{align*}
\enorm{U_{\ell_{k_j}}}^2 
= a(U_{\ell_{k_j}},U_{\ell_{k_j}}) 
\stackrel{\eqref{eq:discreteform}}=
&\dual{f}{U_{\ell_{k_j}}} - \dual{\KK U_{\ell_{k_j}}}{U_{\ell_{k_j}}}
 \\&\xrightarrow{j\to\infty}
\dual{f}{w} - \dual{\KK w}{w} 
\stackrel{\eqref{eq:limitform}}= a(w,w)= \enorm{w}^2.
\end{align*}

\textit{Step 4: Estimator reduction principle (subsequence).}\quad
Let $(\eta_{\ell_{k_j}})_{j\in\N_0}$ denote the estimator subsequence corresponding to $(U_{\ell_{k_j}})_{j\in\N_0}$.
With $\TT_{\ell_{k_{j+1}}}\in\refine(\TT_{{\ell_{k_j}}+1})$ and Lemma~\ref{lemma:reduction}, it holds $\eta_{\ell_{k_{j+1}}}^2 \le \qest\,\eta_{\ell_{k_j}}^2 + \Cest\,\norm{U_{\ell_{k_{j+1}}}-U_{\ell_{k_j}}}{\HH}^2$. Moreover, step~3 implies convergence $\norm{U_{\ell_{k_{j+1}}}-U_{\ell_{k_j}}}{\HH}\simeq\enorm{U_{\ell_{k_{j+1}}}-U_{\ell_{k_j}}}\to0$ as $j\to\infty$. 
Hence, the subsequence $(\eta_{\ell_{k_j}})_{j\in\N_0}$ is contractive up to a sequence that converges to zero.
Therefore, basic calculus (see, e.g.,~\cite[Lemma~2.3]{estconv}) proves convergence $\eta_{\ell_{k_j}} \to0$ as $j\to\infty$. 

\textit{Step 5: Estimator convergence (full sequence).}\quad
We have shown that each subsequence $(\eta_{\ell_k})_{k\in\N_0}$ of $(\eta_\ell)_{\ell\in\N_0}$ has a further subsequence 
$(\eta_{\ell_{k_j}})_{j\in\N_0}$ with $\eta_{\ell_{k_j}}\to0$ as $j\to\infty$. As noted above, this yields $\eta_\ell\to0$ as $\ell\to\infty$.

\textit{Step 6:  Strong convergence of discrete solutions (full sequence).}\quad
\oooldRevision{Finally, reliability~\eqref{axiom:reliability:new}} yields $\norm{u-U_{\ell}}{\HH}\lesssim\eta_\ell\to0$ as $\ell\to\infty$ and hence concludes the proof.
\end{proof}

\oldRevision{\begin{remark}\label{remark:brandnew}
Note that the proof of Proposition~\ref{lemma1:convergence} relies only on~\eqref{axiom:reliability}--\eqref{axiom:infty} to prove boundedness of the estimator sequence $(\eta_\ell)_{\ell\in\N_0}$ (see step~1 of the proof). Instead, we can also modify the marking step~{\rm(iv)} of Algorithm~\ref{algorithm} so that the assertion of Proposition~\ref{lemma1:convergence} remains true, if \eqref{axiom:stability}--\eqref{axiom:reliability:new} still hold, while \eqref{axiom:reliability}--\eqref{axiom:infty} fail. To this end, consider the following new marking criterion:
\begin{itemize}
\item[\revision{\rm(iv)}] If $\eta_\ell > \max_{j=0,\dots,\ell-1}\eta_j$, define $\MM_\ell:=\TT_\ell$. Otherwise, determine a set $\MM_\ell\subseteq\TT_\ell$ of up to the multiplicative constant $\Cmark$ minimal cardinality such that $\theta\eta_\ell^2\le \eta_\ell(\MM_\ell)^2$.
\end{itemize}
To see that this new marking criterion ensures that $(\eta_\ell)_{\ell\in\N_0}$ is bounded, we argue as follows:
\begin{itemize}

\item[]\hspace*{-5mm}\emph{Case 1:} Suppose that there exists an $M\in\N$ such that $\eta_\ell\le\max_{j=0,\dots,\ell-1}\eta_j$ for all $\ell\ge M$. Then, it even follows that $\eta_\ell \le \max_{j=0,\dots,M-1}\eta_j$ for all $\ell\in\N_0$.

\item[]\hspace*{-5mm}\emph{Case 2:} If the assumption of case~1 fails, the new step~{\rm(iv)} of Algorithm~\ref{algorithm} enforces infinitely many steps of uniform refinement. Therefore, Corollary~\ref{cor:uniform} applies and provides $m\in\N_0$ and $C>0$ such that \ooldRevision{all discrete subspaces $\XX_\star\subseteq\HH$ with $\XX_\star\supseteq\XX_m$ admit} a unique solution $U_\star\in\XX_\star$ of~\eqref{eq:discreteform} which is quasi-optimal in the sense of inequality~\eqref{eq:sauterschwab_cea}. Since~\eqref{axiom:stability}--\eqref{axiom:reliability:new} hold, \cite[Lemma~3.5]{axioms} applies and proves quasi-monotonicity of the estimator, i.e., 
\begin{align*}
 \eta_\bullet \le \Cmon\,\eta_\star 
 \quad\text{for all }\TT_\star\in\refine(\TT_m)
 \text{ and all }\TT_\bullet\in\refine(\TT_\star).
\end{align*}
In particular, this implies $\eta_\ell \le \Cmon\,\eta_m$ for all $\ell\ge m$, and therefore $\eta_\ell \le \max\{\Cmon,1\}\max_{j=0,\dots,m}\eta_j$
for all $\ell\in\N_0$.
\end{itemize}
Note that besides step~1 all steps of the proof of Proposition~\ref{lemma1:convergence} rely only on~\eqref{axiom:stability}--\eqref{axiom:reliability:new}. Therefore, we obtain $\eta_\ell\to0$ as $\ell\to\infty$. In particular, this implies that Case~1 above is the generic case and that optimal convergence rates will not be affected by the new marking strategy.
\end{remark}}%

\subsection{Definiteness on the ``discrete'' limit space~(\ref{axiom:infty})}
\label{section:marking}
\oldRevision{While~\eqref{axiom:stability}--\eqref{axiom:reliability} only rely on the a~posteriori error estimation strategy, the property~\eqref{axiom:infty} involves the ``discrete'' limit space $\XX_\infty =\overline{\bigcup_{\ell=0}^\infty\XX_\ell}$ generated by Algorithm~\ref{algorithm} and is hence less accessible for the numerical analysis. 
However, recall that $\HH=\XX_\infty$ is sufficient to ensure~\eqref{axiom:infty}.}
For $\HH=H^1_0(\Omega)$, the
following lemma provides a simple criterion for the latter identity.

\def\DD{\mathcal D}
\begin{lemma}\label{lemma2:infty}
Let $\HH=H^1_0(\Omega)$ and $\XX_\ell =\SSS^p_0(\TT_\ell)$ for some $p\ge1$. 
Suppose that the triangulations $\TT_\ell$ generated by Algorithm~\ref{algorithm} are uniformly shape regular with $\norm{h_\ell}{L^\infty(\Omega)}\to0$ as $\ell\to\infty$. Then, $\XX_\infty = \HH$ and hence assumption~\oooldRevision{\eqref{axiom:infty}} is satisfied.
\end{lemma}%

\begin{proof}
For $w\in\DD:=H^2(\Omega)\cap H^1_0(\Omega)$, recall the approximation property $\inf_{V_\ell\in\XX_\ell}\norm{w-V_\ell}\HH \lesssim \norm{h_\ell}{L^\infty(\Omega)}\norm{D^2w}{L^2(\Omega)}$ from, e.g.,~\cite{brennerscott}. This proves
\begin{align}\label{eq:approximation_property}
 \lim_{\ell\to\infty}\inf_{V_\ell\in\XX_\ell}\norm{w-V_\ell}\HH=0
 \quad\text{for all }w\in\DD.
\end{align}
Let $v\in\HH$ and $\eps>0$. Since $\DD$ is dense within $H^1_0(\Omega)$, choose $w\in\DD$ with $\norm{v-w}\HH\le\eps/2$. According to~\eqref{eq:approximation_property}, there exists an index $\ell_\star\in\N_0$ such that $\inf_{V_\ell\in\XX_\ell}\norm{w-V_\ell}\HH\le\eps/2$ for all $\ell\ge\ell_\star$. In particular, the triangle inequality concludes
\begin{align*}
\inf_{V_\ell\in\XX_\ell}\norm{v-V_\ell}\HH
\le \norm{v-w}\HH + \inf_{V_\ell\in\XX_\ell}\norm{w-V_\ell}\HH\le\eps
\quad\text{for all }\ell\ge\ell_\star.
\end{align*}
This proves $v\in\XX_\infty= \overline{\bigcup_{\ell=0}^\infty\XX_\ell}$ and hence concludes $\XX_\infty  = \HH$.
\end{proof}%

The following proposition shows that  $\norm{h_\ell}{L^\infty(\Omega)}\to0$ and hence~\oooldRevision{\eqref{axiom:infty}} with $\XX_\infty=\HH$ is automatically verified in many generic situations. In particular, we note that~\eqref{eq:eff-rel} is well-known for residual error estimators and elliptic PDEs with polynomial coefficients.

\def\PP{\mathcal P}
\def\Ceff{C_{\rm eff}}
\begin{proposition}\label{prop:A:infty}
Suppose~\oooldRevision{\eqref{axiom:stability}--\eqref{axiom:reliability:new}} and $0<\theta\le1$. Employ the notation of Algorithm~\ref{algorithm}. 
Let $p\ge1$ and $q\ge0$ be polynomial degrees.
Suppose that $\HH = H^1_0(\Omega)$ and $\XX_\ell =\SSS^p_0(\TT_\ell)$. For $f\in L^2(\Omega)$, let $f_\ell\in\PP^{q}(\TT_\ell):=\set{V_\ell\in L^2(\Omega)}{\forall T\in\TT_\ell\quad V_\ell|_T\text{ is a polynomial}\linebreak \text{of degree }\le q}$ be the $L^2$-best approximation of $f$ in $\PP^q(\TT_\ell)$. Suppose that the error estimator is \oooldRevision{even} 
reliable in the sense of 
\begin{align}\label{eq:eff-rel}
\begin{split}
 \norm{u-U_\ell}{H^1(\Omega)} + \norm{h_\ell(f-f_\ell)}{L^2(\Omega)}
 \le\Crel\,\eta_\ell
 \quad\text{for all }\ell\ge0,
\end{split}
\end{align}
where $\Crel>0$ is independent of $\ell$. Suppose that for all $\ell\in\N$ and all $T\in\TT_\ell$ it holds $u|_T\not\in\PP^p(T)$ or $f|_T\not\in\PP^q(T)$, i.e., the continuous solution or the given data are not locally polynomial. Then, Algorithm~\ref{algorithm} implies
convergence $\norm{h_\ell}{L^\infty(\Omega)}\to0$ as $\ell\to\infty$. In particular, assumption~\oooldRevision{\eqref{axiom:infty}} is satisfied \oooldRevision{with $\XX_\infty=\HH$}.
\end{proposition}%

\begin{proof}
We argue by contradiction and suppose that $\norm{h_\ell}{L^\infty(\Omega)}\ge C>0$ for some $C>0$ and all $\ell\in\N_0$. Since the meshes $\TT_\ell$ are obtained by successive refinement, there exists some index $\ell_\star\in\N_0$ and some element $T\in\TT_{\ell_\star}$ which remains unrefined, i.e., $T\in\TT_\ell$ for all $\ell\ge\ell_\star$. Proposition~\ref{lemma1:convergence} yields $\eta_\ell\to0$ as $\ell\to\infty$. In particular, we infer 
\begin{align*}
 \norm{u-U_\ell}{L^2(T)} + \norm{f-f_\ell}{L^2(T)}\xrightarrow{\ell\to\infty}0.
\end{align*}
Since $U_\ell\in\PP^p(T)$ and $f_\ell\in\PP^q(T)$, we conclude that $u\in\PP^p(T)$ and $f\in\PP^q(T)$. This, however, contradicts the assumptions on $u$ and $f$. Overall, we thus obtain $\norm{h_\ell}{L^\infty(\Omega)}\to0$ as $\ell\to\infty$, and Lemma~\ref{lemma2:infty} concludes the proof.
\end{proof}%

\oldRevision{The next proposition shows} that $\norm{h_\ell}{L^\infty(\Omega)}\to0$ and hence~\oooldRevision{\eqref{axiom:infty}} with $\XX_\infty=\HH$ can also be guaranteed by \oldRevision{employing an} \oooldRevision{expanded D\"orfler} marking strategy in \oooldRevision{\oldRevision{step~{\rm(iv)}}} of Algorithm~\ref{algorithm}. We stress that this does not affect optimal convergence behaviour \oldRevision{in the sense of Theorem~\ref{theorem:optimal} below.}

\begin{proposition}\label{prop:doerfler}
\oooldRevision{Suppose~$0<\theta\le1$}. Employ the notation of Algorithm~\ref{algorithm}. Let $\Cmark'>0$. For all $\ell\in\N_0$, we suppose that the set $\MM_\ell\subseteq\TT_\ell$ in \oooldRevision{\oldRevision{step~{\rm(iv)}}} of Algorithm~\ref{algorithm} is selected as follows: 
\begin{itemize}
\item Let $\MM_\ell'\subseteq\TT_\ell$ be a set of up to the multiplicative constant $\Cmark'$ minimal cardinality such that $\theta\eta_\ell^2 \le \eta_\ell(\MM_\ell')^2$. 
\item Suppose that $\TT_\ell = \{T_1,\dots,T_N\}$ is sorted such that $|T_1| \ge |T_2| \ge \dots \ge |T_N|$.
\item With arbitrary $1\le n\le\#\MM_\ell'$, define $\MM_\ell := \MM_\ell' \cup \{T_1,\dots,T_n\}$.
\end{itemize}
Then, $\MM_\ell\subseteq\TT_\ell$ is a set of up to the multiplicative constant $\Cmark:=2\Cmark'$ minimal cardinality such that $\theta\eta_\ell^2 \le \eta_\ell(\MM_\ell)^2$. Moreover, Algorithm~\ref{algorithm} guarantees $\norm{h_\ell}{L^\infty(\Omega)}\to0$ as $\ell\to\infty$. 
In particular, assumption~\oooldRevision{\eqref{axiom:infty}} \oooldRevision{with $\XX_\infty=\HH$} is satisfied
 for $\HH=H^1_0(\Omega)$ and $\XX_\ell =\SSS^p_0(\TT_\ell)$.
\end{proposition}%

\begin{proof}
The claims on $\MM_\ell$ are obvious. Recall that refinement leads to a uniform contraction of the mesh-size, i.e., $h_{\ell+1}|_T \le \qmesh h_\ell|_T$ for all $T\in\MM_\ell\subseteq\TT_\ell\backslash\TT_{\ell+1}$. Since each mesh $\TT_\ell$ is a finite set and each step of the adaptive algorithm guarantees that (at least) the element $T\in\TT_\ell$ with the largest size $|T| \simeq (h_\ell|_T)^d$ is refined, this implies necessarily $\norm{h_\ell}{L^\infty(\Omega)}\to0$ as $\ell\to\infty$. Lemma~\ref{lemma2:infty} concludes the proof.
\end{proof}%

\subsection{Linear convergence of adaptive algorithm}
\label{section:convergence:linear}
The analysis in this section adapts and extends some ideas from~\cite{ffp}. We note that the latter work uses strong ellipticity~\eqref{eq:ffp} of $b(\cdot,\cdot)$, while we only rely on ellipticity~\eqref{eq:elliptic} of $a(\cdot,\cdot)$.

\begin{lemma}[{\cite[Lemma~3.5]{ffp}}]\label{lemma3:convergence}
Suppose~\eqref{axiom:stability}--\oooldRevision{\eqref{axiom:infty}} and $0<\theta\le1$. Employ the notation of Algorithm~\ref{algorithm}.
Then, the sequences $(e_\ell)_{\ell \in \N}$ and $(E_\ell)_{\ell \in \N}$ defined by
\begin{align*}
	e_\ell &:= \begin{cases}
		\frac{u-U_\ell}{\norm{u-U_\ell}{\HH}} \quad &\text{for} \,\, u \neq U_\ell, \\
		0 \quad &\text{else},
	\end{cases} \\
	E_\ell &:= \begin{cases}
		\frac{U_{\ell +1}-U_\ell}{\norm{U_{\ell+1}-U_\ell}{\HH}} \quad &\text{for} \,\, U_{\ell+1} \neq U_\ell, \\
		0 \quad &\text{else},
	\end{cases}
\end{align*}
converge weakly to zero, i.e.,
$\lim\limits_{\ell \rightarrow \infty} \dual{\phi}{e_\ell} = 0 = \lim\limits_{\ell \rightarrow \infty} \dual{\phi}{E_\ell}$
\text{for all} $\phi \in \HH^*$.
\end{lemma}%

\begin{proof}
We consider the sequence $(e_\ell)_{\ell\in\N_0}$ and note that the claim for $(E_\ell)_{\ell\in\N_0}$ follows along the same lines. To prove $e_\ell\rightharpoonup0$ as $\ell\to\infty$, we show that each subsequence $(e_{\ell_k})_{k\in\N_0}$ admits a further subsequence $(e_{\ell_{k_j}})_{j\in\N_0}$ such that $e_{\ell_{k_j}}\rightharpoonup0$ as $j\to\infty$.
Let $(e_{\ell_k})_{k\in\N_0}$ be a subsequence of $(e_\ell)_{\ell\in\N_0}$. Due to boundedness $\norm{e_{\ell_k}}\HH\le1$, there exists a further subsequence $(e_{\ell_{k_j}})_{j\in\N_0}$ such that $e_{\ell_{k_j}}\rightharpoonup w\in\HH$ as $j\to\infty$. It remains to show that $w=0$. Note that $U_\ell,u\in\XX_\infty$ \oldRevision{(see Proposition~\ref{lemma1:convergence})} implies $e_\ell\in\XX_\infty$ and hence $w\in\XX_\infty$. 
Note the Galerkin orthogonality
\begin{align}\label{eq:galerkin:orthogonality}
	0 = b(u-U_\star,V_\star) = a(u-U_{\star},V_\star) + \dual{\KK(u - U_\star)}{V_\star} \quad \text{for all} \,\, V_\star \in \XX_\star.
\end{align}
Let $n\in\N$ and $V_n\in\XX_n$. If $\ell_{k_j}\ge n$ and $e_{\ell_{k_j}}\neq0$, the Galerkin orthogonality proves
\begin{align*}
 b(e_{\ell_{k_j}},V_n) = b(u-U_{\ell_{k_j}},V_n)/\norm{u-U_{\ell_{k_j}}}\HH = 0
\end{align*}
and hence $b(e_{\ell_{k_j}},V_n) = 0$ for all $\ell_{k_j}\ge n$. With weak convergence, this yields
\begin{align*}
 b(w,V_n) = \lim_{j\to\infty}b(e_{\ell_{k_j}},V_n) = 0
 \quad\text{for all $V_n\in\XX_n$ and all $n\in\N_0$.}
\end{align*}
Let $v\in\XX_\infty$.
By definition of $\XX_\infty$, there exists a sequence $(V_n)_{n\in\N_0}$ with $V_n\in\XX_n$ and $\norm{v-V_n}\HH \to 0$ as $n\to\infty$.
Therefore the preceding identity implies $b(w,v)=0$ for all $v\in\XX_\infty$. Finally, assumption~\oooldRevision{\eqref{axiom:infty}} concludes $w=0$.
\end{proof}%

The following quasi-orthogonality~\eqref{eq:orthogonality} is a consequence of Lemma~\ref{lemma3:convergence} and the Galerkin orthogonality~\eqref{eq:galerkin:orthogonality}. For elliptic $b(\cdot,\cdot)$, it is proved
in~\cite[Proposition~3.6]{ffp}. Our proof essentially follows those ideas, but we use the norm $\enorm\cdot$ induced by $a(\cdot,\cdot)$ instead of the quasi-norm induced by $b(\cdot,\cdot)$, if $b(\cdot,\cdot)$ was elliptic. 
For the convenience of the reader, we include the most important steps of the proof.

\begin{lemma}
\label{lemma2:convergence}
Suppose~\eqref{axiom:stability}--\oooldRevision{\eqref{axiom:infty}} and $0<\theta\le1$. Employ the notation of Algorithm~\ref{algorithm}. Then, for any $0<\eps<1$, there exists $\ell_3\in\N_0$ such that
\begin{align}\label{eq:orthogonality}
 \enorm{u-U_{\ell+1}}^2 + \enorm{U_{\ell+1}-U_\ell}^2 \le \frac{1}{1-\eps}\,\enorm{u-U_\ell}^2
 \quad\text{for all $\ell\ge\ell_3$.}
\end{align}
\end{lemma}

\begin{proof}
Let $\eps>0$. Let $\delta>0$ be a free parameter which is fixed later.
Consider the sequences $(e_\ell)_{\ell\in\N_0}$ and $(E_\ell)_{\ell\in\N_0}$ of Lemma~\ref{lemma3:convergence}.
Recall that the compact operator $\KK$ turns weak convergence $e_\ell,E_\ell \rightharpoonup 0$ in $\HH$ into strong convergence
$\KK e_\ell, \KK E_\ell \rightarrow 0$ in $\HH^*$ as $\ell \rightarrow \infty$. 
For any $\delta >0$, this provides some $\ell_3 \in \N$ such that 
\begin{align*}
 \norm{\KK e_{\ell}}{\HH^*} + \norm{\KK E_{\ell}}{\HH^*} \le \delta
 \quad\text{for all }\ell\ge\ell_3.
\end{align*}
For any $w\in\HH$, this gives
\begin{align*}
	|\dual{\KK (u - U_\ell )}{w}| &= |\dual{\KK e_\ell}{w}| \, \norm{u - U_\ell}{\HH}
	\leq \delta\, \norm{u - U_\ell}{\HH} \norm{w}{\HH}
\end{align*}
as well as
\begin{align*}
	|\dual{\KK (U_{\ell+1} - U_\ell )}{w}| &= |\dual{\KK E_{\ell}}{w}| \, \norm{U_{\ell+1} - U_\ell}{\HH}
		\leq \delta\,  \norm{U_{\ell+1} - U_\ell}{\HH}\norm{w}{\HH}.
\end{align*}
Algebraic computations with the Galerkin orthogonality~\eqref{eq:galerkin:orthogonality} show
\begin{align*}
 b(u\!-\!U_{\ell+1},u\!-\!U_{\ell+1})
 + b(U_{\ell+1}\!-\!U_\ell,U_{\ell+1}\!-\!U_\ell)
 + b(U_{\ell+1} \!-\! U_\ell,u \!-\! U_{\ell+1}) 
 = b(u\!-\!U_\ell,u\!-\!U_\ell).
\end{align*}
Since $\enorm{v}^2 = a(v,v) = b(v,v)-\dual{\KK v}{v}$ for all $v\in\HH$, this translates to
\begin{align*}
 &\enorm{u-U_{\ell+1}}^2 + \enorm{U_{\ell+1}-U_\ell} ^2
 + \dual{\KK(u-U_{\ell+1})}{u-U_{\ell+1}}
 + \dual{\KK(U_{\ell+1}-U_\ell)}{U_{\ell+1}-U_\ell}
 \\&\quad
 + b(U_{\ell+1} - U_\ell,u - U_{\ell+1})
 = \enorm{u-U_{\ell}}^2 + \dual{\KK(u-U_\ell)}{u-U_\ell}.
\end{align*}
The remaining bilinear form $b(U_{\ell+1} - U_\ell,u - U_{\ell+1})$ is estimated as follows
\begin{eqnarray*}
	|b(U_{\ell+1} - U_\ell,u - U_{\ell+1})| & =& |\overline{a(u - U_{\ell+1},U_{\ell+1} - U_\ell)}  + \dual{\KK(U_{\ell+1} - U_\ell)}{u - U_{\ell+1}}| \notag\\
	& \stackrel{\eqref{eq:galerkin:orthogonality}}{=}& 
	 | -\overline{\dual{\KK(u - U_{\ell+1})}{U_{\ell+1} - U_\ell}} + \dual{\KK(U_{\ell+1} - U_\ell)}{u - U_{\ell+1}}| \notag\\
	 & \leq& 2\delta\, \norm{u - U_{\ell+1}}{\HH} \norm{U_{\ell+1} - U_\ell}{\HH}.
\end{eqnarray*}
With norm equivalence $\norm{v}\HH^2 \le C\,\enorm{v}^2$ for all $v\in\HH$, we thus see
\begin{align*}
 &(1-\delta C)\,\enorm{u-U_{\ell+1}}^2
 +(1-\delta C)\,\enorm{U_{\ell+1}-U_\ell} ^2
 \\&\quad
 \le (1+\delta C)\,\enorm{u-U_\ell}^2
 + 2\delta C\,\enorm{u - U_{\ell+1}} \enorm{U_{\ell+1} - U_\ell}.
\end{align*}
Finally, the Young inequality $2cab \le ca^2 + cb^2$ for all $a,b,c\ge0$, yields
\begin{align*}
 (1-2\delta C)\,\enorm{u-U_{\ell+1}}^2
 + (1-2\delta C)\,\enorm{U_{\ell+1}-U_\ell} ^2
 \le (1+\delta C)\,\enorm{u-U_\ell}^2.
\end{align*}
For sufficiently small $\delta>0$ and $\displaystyle\frac{1+\delta C}{1-2\delta C} \le \frac{1}{1-\eps}$, this proves~\eqref{eq:orthogonality}.
\end{proof}%

\def\qlin{q_{\rm lin}}
\def\Clin{C_{\rm lin}}
The following result was proved in~\cite{ffp} for strongly elliptic problems~\eqref{eq:ffp}. Here, we generalize the result by extending it to a more general class of problems. Our proof follows the ideas of~\cite{ckns}. 

\begin{theorem} \label{theorem:convergence}
Suppose~\eqref{axiom:stability}--\oooldRevision{\eqref{axiom:infty}} and $0<\theta\le1$. Then,
there exist constants $0<\qlin<1$ and $\Clin>0$ such that the output of Algorithm~\ref{algorithm} satisfies $\eta_{\ell+n}\le\Clin\qlin^n\,\eta_\ell$ for all $\ell,n\in\N_0$ \oooldRevision{with $\ell\ge\ell_3$}, \ooldRevision{where $\ell_3\in\N_0$ is the index from Lemma~\ref{lemma2:convergence}.}
\end{theorem}

\begin{proof}
Due to norm equivalence $\norm\cdot\HH\simeq\enorm\cdot$, \oooldRevision{reliability~\eqref{axiom:reliability:new} and estimator reduction (Lemma~\ref{lemma:reduction})} also hold with respect to the $a(\cdot,\cdot)$-induced energy norm $\enorm\cdot$. To simplify the notation and without loss of generality, we therefore suppose $\norm\cdot\HH = \enorm\cdot$ throughout the proof.

\textit{Step~1:}
\quad In this step, we prove that there exist $0<\qlin,\ooldRevision{\lambda}<1$ and $\ell_3\in\N_0$ such that
\begin{align}\label{eq1:lemma2}
 \Delta_{\ell+1} \le \qlin\,\Delta_\ell
 \quad\text{for all }\ell\ge\ell_3,
 \quad\text{where}\quad
 \Delta_\star^2 := \enorm{u-U_\star}^2 + \ooldRevision{\lambda}\,\eta_\star^2.
\end{align}
Let $\eps,\ooldRevision{\lambda}>0$ be free parameters which are fixed later. With Lemma~\ref{lemma:reduction} and Lemma~\ref{lemma2:convergence}, we see for $\ell\ge\ell_3=\ell_3(\eps)$
\begin{align*}
 \Delta_{\ell+1}^2 &= \enorm{u-U_{\ell+1}}^2 + \ooldRevision{\lambda}\,\eta_{\ell+1}^2
 \le \frac{1}{1-\eps}\,\enorm{u-U_\ell}^2 + \ooldRevision{\lambda}\,\qest\,\eta_\ell^2+(\ooldRevision{\lambda}\Cest-1)\,\enorm{U_{\ell+1}-U_\ell}^2.
\end{align*}
For sufficiently small $\ooldRevision{\lambda}$ (i.e., $\ooldRevision{\lambda}\Cest\le1$) and an additional free parameter $\delta>0$, \oooldRevision{reliability~\eqref{axiom:reliability:new}} yields that
\begin{align*}
 \Delta_{\ell+1}^2
 \le \frac{1}{1-\eps}\,\enorm{u-U_\ell}^2 + \ooldRevision{\lambda}\,\qest\,\eta_\ell^2
 &\le\Big(\frac{1}{1-\eps}-\delta\ooldRevision{\lambda}\Big)\,\enorm{u-U_\ell}^2 + \ooldRevision{\lambda}(\qest+\Crel^2\delta)\,\eta_\ell^2
 \\&
 \le\max\Big\{\frac{1}{1-\eps}-\delta\ooldRevision{\lambda}\,,\,\qest+\Crel^2\delta\Big\}\,\Delta_\ell^2.
\end{align*}
Since $0<\qest<1$, we may choose $\delta>0$ sufficiently small such that $0<\qest+\Crel^2\delta<1$. Finally choose $\eps>0$ sufficiently small such that $0<1/(1-\eps)-\delta\ooldRevision{\lambda}<1$. This concludes~\eqref{eq1:lemma2}.

\textit{Step~2:}
We employ the notation of step~1. Induction on $n$ proves $\Delta_{\ell+n}\le\qlin^n\,\Delta_\ell$ for all $\ell\ge\ell_3$ and all $n\in\N_0$. Note that \oooldRevision{reliability~\eqref{axiom:reliability:new}} yields $\eta_\star^2 \simeq \Delta_\star^2$. Combining these two observations, we conclude the proof.
\end{proof}

\subsection{Validity of the C\'ea lemma}
\label{section:cea}%
In this section, we show that the discrete solutions computed in Algorithm~\ref{algorithm} are quasi-optimal in the sense of the C\'ea lemma.

\begin{theorem}\label{prop:cea}
Suppose~\eqref{axiom:stability}--\oooldRevision{\eqref{axiom:infty}} and $0<\theta\le1$. Then,
there exist $C_\ell\ge1$ with $\lim\limits_{\ell\to\infty}C_\ell = 1$ and $\ell_4>0$ such that
the output of Algorithm~\ref{algorithm} satisfies
\begin{align}\label{eq:cea}
 \enorm{u-U_\ell} \le C_\ell\,\min_{V_\ell\in\XX_\ell}\enorm{u-V_\ell}
 \quad\text{for all }\ell\ge\ell_4.
\end{align}
\end{theorem}%

\begin{proof}
Consider the sequences $(e_\ell)$ and $(E_\ell)$ of Lemma~\ref{lemma3:convergence}. 
We follow the arguments of the proof of Lemma~\ref{lemma2:convergence}.
Let $V_\ell\in\XX_\ell$.
With the Galerkin orthogonality~\eqref{eq:galerkin:orthogonality}, it holds
\begin{eqnarray*}
 \enorm{u-U_\ell}^2
 &\!=\!& b(u-U_\ell,u-U_\ell) - \dual{\KK(u-U_\ell)}{u-U_\ell}
 \\
 &\!\!\!\stackrel{\eqref{eq:galerkin:orthogonality}}=\!\!\!& b(u-U_\ell,u-V_\ell) - \dual{\KK(u-U_\ell)}{u-U_\ell}
 \\
 &\!=\!& a(u-U_\ell,u-V_\ell) + \dual{\KK(u-U_\ell)}{u-V_\ell} - \dual{\KK(u-U_\ell)}{u-U_\ell}
 \\ 
 &\!\le\!& \enorm{u-U_\ell}\enorm{u-V_\ell}
 + \norm{\KK e_\ell}{\HH^*}\norm{u-U_\ell}\HH\norm{u-V_\ell}\HH+\norm{\KK e_\ell}{\HH^*}\norm{u-U_\ell}\HH^2.
\end{eqnarray*}
With norm equivalence $\norm{v}\HH^2 \le C\,\enorm{v}^2$ for all $v\in\HH$, we thus see
\begin{align*}
 \enorm{u-U_\ell}
 \le (1+C\,\norm{\KK e_\ell}{\HH^*})\,\enorm{u-V_\ell} + C\,\norm{\KK e_\ell}{\HH^*}\enorm{u-U_\ell}.
\end{align*}
Rearranging this estimate, we prove
\begin{align*}
 \enorm{u-U_\ell} \le \frac{1+C\,\norm{\KK e_\ell}{\HH^*}}{1-C\,\norm{\KK e_\ell}{\HH^*}}\,\enorm{u-V_\ell}
\end{align*}
and conclude~\eqref{eq:cea}, since $\norm{\KK e_\ell}{\HH^*}\to0$ as $\ell\to\infty$. 
\end{proof}%

\section{Optimal Convergence Rates}
\label{section:rates}

\subsection{Fine properties of mesh-refinement}
\label{section2:mesh}%
The proof of optimal convergence rates requires further properties of the mesh-refinement. First, we suppose
that each refined element is split in at most $\Cson$ and at least 2 sons. In particular, it holds
\begin{align}\label{mesh:sons}
\# (\TT_\star \setminus \TT_\plus) + \# \TT_\star \leq \# \TT_\plus \quad \text{for all} \,  \TT_\star \in \T \text{ and all } \TT_\plus \in \refine(\TT_\star).
\end{align}
Second, we require the mesh-closure estimate
\begin{align}\label{mesh:mesh-closure}
\# \TT_\ell - \# \TT_0 \leq \Cmesh \sum_{j =0}^{\ell -1} \# \MM_j \quad \text{for all} \, \ell \in \N,
\end{align}
where the constant $\Cmesh\ge1$ depends only on the initial mesh $\TT_0$. 
Finally, we need the overlay estimate, i.e., for all triangulations $\TT\in\T$ and all $ \TT_\plus, \TT_\star \in \refine(\TT)$ there
exists a common refinement $\TT_\plus  \oplus \TT_\star \in \refine(\TT_\plus) \cap \refine(\TT_\star)\subseteq\refine(\TT)$ which satisfies
\begin{align}\label{mesh:overlay}
 \# (\TT_\plus \oplus \TT_\star) \leq \# \TT_\plus + \# \TT_\star - \# \TT.
\end{align}
For newest vertex bisection (NVB), the mesh-closure estimate has first been proved for $d=2$ in~\cite{bdd} and later for $d \geq 2$ in \cite{stevenson08}. 
While both works require an additional admissibility assumption on $\TT_0$, \cite{kpp} proved that this condition is unnecessary for $d=2$. 
The proof of the overlay estimate is found in \cite{ckns,stevenson07}. 
We note that NVB ensures $2\le\Cson<\infty$, where $\Cson$ depends only on $\TT_0$ and $d$; see~\cite{gss}. For $d=2$, it holds $\Cson=4$ (see, e.g.,~\cite{kpp}). For other mesh-refinement strategies than NVB which satisfy~\eqref{mesh:sons}--\eqref{mesh:overlay}, we refer to~\cite{bn,mp} as well as to~\cite[Section~2.5]{axioms}.

\begin{lemma}\label{lemma:m}
\ooldRevision{NVB guarantees the following \oldRevision{properties~{\rm(a)--(c)}} which are exploited in our analysis of optimal convergence rates:}
\begin{itemize}
\item[\oldRevision{\rm(a)}] There exists $m \in \N$ such that the $m$-times uniform refinement $\widehat\TT_0$ of $\TT_0$ satisfies the \oooldRevision{assertions of Lemma~\ref{prop:new} (with \ooldRevision{$\TT_{\ell_2}$} replaced by $\widehat\TT_0$). In particular, there holds the quasi-monotonicity of the estimator, i.e., \ooldRevision{there exists an independent constant $\Cmon>0$ such that}
\begin{align*}
 \eta_\bullet \le \Cmon\,\eta_\star
 \ooldRevision{\quad\text{for all }\TT_\star\in\T
 \text{ and all }\TT_\bullet\in\refine(\widehat\TT_0)\cap\refine(\TT_\star),}
\end{align*}
\ooldRevision{provided that the Galerkin solution $U_\star\in\XX_\star$ exists}.}
\item[\oldRevision{\rm(b)}]
Moreover, for all $\TT_\star\in\T$, the $m$-times uniform refinement $\widehat\TT_\star$ of $\TT_\star$ guarantees $\widehat\TT_\star\in\refine(\widehat\TT_0)$ and $\# \widehat{\TT}_\star \leq  \Cson^m \# \TT_\star$.
\item[\oldRevision{\rm(c)}] \ooldRevision{Suppose that $\XX_\infty = \overline{\bigcup_{\ell=0}^\infty\XX_\ell} = \HH$ (e.g., the expanded D\"orfler marking strategy from Proposition~\ref{prop:doerfler} is used). Then, there exists an index $\ell_5\in\N_0$ such that $\TT_{\ell}\in\refine(\widehat\TT_0)$ for all $\ell\ge\ell_5$.}
\end{itemize}
\end{lemma}%

\begin{proof}
\oldRevision{Assertion~(a) is} a direct consequence of Corollary~\ref{cor:uniform}, 
\oooldRevision{if we argue as in the proof of Lemma~\ref{prop:new}.} \oldRevision{Assertions~(b)--(c)} follow from the fact that NVB is based on a binary refinement rule, where the order of the refinements does not matter~\cite{stevenson08}.
\end{proof}%

\subsection{Approximation classes}
\label{section:classes}
For $N\in\N_0$ and $\TT\in\T$, we define 
\begin{align}\label{def:TN}
\T_N(\TT) := \set{\TT_\star \in \refine(\TT)}{\#\TT_\star - \# \TT \leq N\text{ and solution $U_\star\in\XX_\star$ to~\eqref{eq:discreteform} exists}}.
\end{align}
We note that $\T_N(\TT)$ is finite, but may be empty. However, according to Lemma~\ref{lemma:m}, it holds $\T_N(\TT)\neq\emptyset$ for all sufficiently large $N$, e.g., $N\ge\Cson^m\#\TT$. We use the convention $\min_{\TT_\star \in \T_N(\TT)} \eta_\star = 0$, if $\T_N(\TT)=\emptyset$. For $s>0$, we then define
\begin{align}\label{eq:approximatoin_class}
	\norm{u}{\A_s(\TT)} := \sup_{N \in \N_0} \Big( (N+1)^s \min_{\TT_\star \in \T_N(\TT)} \eta_\star \Big),
\end{align}
where $\eta_\star$ is the error estimator corresponding to the optimal triangulation $\TT_\star \in \T_N(\TT)$.
Note that $\norm{u}{\A_s(\TT)}<\infty$ means that starting from $\TT$, a convergence behaviour of $\eta_\star = \OO\big((\#\TT_\star)^{-s}\big)$ is possible, if the optimal meshes are chosen. To abbreviate notation, we let
\begin{align}
 \T_N := \T_N(\TT_0)
 \quad\text{and}\quad
 \norm{u}{\A_s} := \norm{u}{\A_s(\TT_0)}.
\end{align}%

\begin{lemma}\label{lemma:noe}
For all $\TT\in\T$ and $\TT_\plus\in\refine(\TT)$, it holds
\begin{align}\label{eq:noe}
 \#\TT_\plus-\#\TT + 1 \le \#\TT_\plus \le \#\TT\,(\#\TT_\plus-\#\TT+1).
\end{align}
\end{lemma}%

\begin{proof}
Note that
$\big(\#\TT_\plus - \#\TT + 1\big)-\#\TT_\plus/\#\TT
 = \big(\#\TT_\plus-\#\TT\big)\big(1-1/\#\TT\big) \ge0.$
Rearranging the terms, we conclude the upper bound in~\eqref{eq:noe}, while the lower bound is obvious.
\end{proof}%

\begin{lemma}\label{lemma:As}
There  
exists $C_3>0$ which depends 
only on $\Cson$, $m$ from Lemma~\ref{lemma:m}, and $\TT_0$, such that for all $s>0$ and all $\TT\in\T$, it holds
\begin{align}\label{eq1:As}
 \sup_{N\ge C_3\#\TT}\Big((N+1)^s\,\min_{\TT_\star\in\T_N}\eta_\star\Big)\le2^s\,\norm{u}{\A_s(\TT)}
\end{align}
as well as 
\begin{align}\label{eq2:As}
 \sup_{N\ge C_3\#\TT}\Big((N+1)^s\,\min_{\TT_\star\in\T_N(\TT)}\eta_\star\Big)
 \le \Cmon\,2^s\,\norm{u}{\A_s}.
\end{align}
In particular, there holds equivalence
\begin{align}\label{eq3:As}
 \norm{u}{\A_s(\TT)} < \infty
 \quad\Longleftrightarrow\quad
 \norm{u}{\A_s} < \infty.
\end{align}
\end{lemma}

\begin{proof}
\emph{Step 1: The estimates~\eqref{eq1:As}--\eqref{eq2:As} imply~\eqref{eq3:As}.}\quad
For any $M>0$, the sets $\bigcup_{N=0}^M\T_N$ and $\bigcup_{N=0}^M\T_N(\TT)$ are finite. Hence, \eqref{eq1:As} provides an upper bound to $\norm{u}{\A_s}$ in terms of $\norm{u}{\A_s(\TT)}$, up to some finite summand which depends on $M=C_3\#\TT-1$. Therefore, $\norm{u}{\A_s(\TT)}<\infty$ implies $\norm{u}{\A_s}<\infty$. The converse implication follows analogously.

\emph{Step 2: Verification of~\eqref{eq1:As}.}\quad
Let $N\ge0$. Apply Lemma~\ref{lemma:m} to see that the $m$-times uniform refinement $\widehat\TT$ of $\TT$ satisfies $\#\TT\le\#\widehat\TT\le \Cson^m\#\TT =: C$ and $\widehat\TT\in\T_C\subseteq\T_{C+N}(\TT)$, i.e., $\T_{C+N}(\TT)\neq\emptyset$. Choose $\TT_\plus\in\T_{C+N}(\TT)$ with $\eta_\plus = \min_{\TT_\star\in\T_{C+N}(\TT)}\eta_\star$. Then, we estimate
\begin{align*}
 \#\TT_\plus-\#\TT_0 = (\#\TT_\plus-\#\TT)+(\#\TT-\#\TT_0) \le (C+N) + \#\TT \le 2C+N,
\end{align*}
i.e., $\TT_\plus \in \T_{2C+N}$. By choice of $\TT_\plus\in\T_{C+N}(\TT)$ and the definition of $\norm{u}{\A_s(\TT)}$, it follows
\begin{align*}
 (2C+N+1)^s\,\min_{\TT_\star\in\T_{2C+N}}\eta_\star\le\Big(\frac{2C+N+1}{C+N+1}\Big)^s\,(C+N+1)^s\,\eta_\plus 
 \le2^s\,\norm{u}{\A_s(\TT)}.
\end{align*}
Since this estimate holds for all $N\ge0$, \oooldRevision{we obtain}~\eqref{eq1:As} with $C_3=2\Cson^m$.

\emph{Step 3: Verification of~\eqref{eq2:As}.}\quad
Let $N\ge0$. Adopt the notation from step~2 and recall that $\widehat\TT\in\T_C\subseteq\T_{C+N}$. Choose $\TT_+\in\T_{C+N}$ with $\eta_+ = \min_{\TT_\star\in\T_{C+N}}\eta_\star$. Define $\TT_\plus:=\widehat\TT\oplus\TT_+$ to ensure that the discrete solution $U_\plus\in\XX_\plus$ exists. Then,
\begin{align*}
 \#\TT_\plus - \#\TT \stackrel{\eqref{mesh:overlay}}\le (\#\widehat\TT+\#\TT_+-\#\TT_0)-\#\TT
 \le \#\widehat\TT + C+N
 \le 2C+N,
\end{align*}
i.e., $\TT_\plus\in\T_{2C+N}(\TT)$. Moreover, \oooldRevision{quasi-monotonicity of the estimator (Lemma~\ref{lemma:m}) yields}
\begin{align*}
 (2C+N+1)^s\,\min_{\TT_\star\in\T_{2C+N}(\TT)}\eta_\star&\le(2C+N+1)^s\eta_\plus 
 \\&\le \Cmon\,\Big(\frac{2C+N+1}{C+N+1}\Big)^s\,(C+N+1)^s\,\eta_+
 \le \Cmon\,2^s\,\norm{u}{\A_s}.
\end{align*}
Since this estimate holds for all $N\ge0$, \oooldRevision{we obtain}~\eqref{eq2:As} again with $C_3=2\Cson^m$.
\end{proof}

\def\osc{{\rm osc}}%
\def\E{\mathbb E}%
In the spirit of~\cite{ckns}, one can also consider approximation classes based on the so-called \emph{total error}. Suppose that the Galerkin solution $U_\star\in\XX_\star$ of~\eqref{eq:discreteform} exists. Suppose that $\osc_\star:\XX_\star\to\R$ are so-called oscillation terms  such that the error estimator is reliable and efficient in the sense of
\begin{align}\label{eq:osc}
 \Crel^{-1}\,\norm{u-U_\star}\HH \le \eta_\star
 \le \Ceff\,\big(\norm{u-U_\star}\HH + \osc_\star(U_\star)\big).
\end{align}
Then,~\cite{ckns} considers
\begin{align}\label{eq2:approximatoin_class}
	\norm{u}{\E_s(\TT)} := \sup_{N \in \N_0} \Big( (N+1)^s \min_{\substack{\TT_\star \in \refine(\TT)\\\#\TT_\star-\#\TT\le N}} \inf_{V_\star\in\XX_\star}\big(\norm{u-V_\star}\HH + \osc_\star(V_\star)\big) \Big)
	\quad\text{for $\TT\in\T$}.
\end{align}
Note that the definition of $\norm{u}{\E_s(\TT)}$ also involves meshes for which the existence of the discrete solution may fail.
Adapting~\cite[Theorem~4.4]{axioms}, we derive the following result which states that the total error (starting from some arbitrary initial mesh $\TT$) converges with the same algebraic rate as the error estimator.

\begin{lemma}\label{lemma:osc}
Let $\osc_\star:\XX_\star\to\R$ satisfy~\eqref{eq:osc}. Suppose that there exists $C_{\rm osc}>0$ such that for all $\TT_\star\in\T$ for which the discrete solution $U_\star\in\XX_\star$ of~\eqref{eq:discreteform} exists, it holds the following:
\begin{itemize}
\item $\osc_\star := \osc_\star(U_\star) \le C_{\rm osc}\,\eta_\star$,
\item $C_{\rm osc}^{-1}\,\osc_\star(V_\star) \le \osc_\star(W_\star) + \norm{V_\star-W_\star}\HH$
for all $V_\star,W_\star\in\XX_\star$.
\end{itemize}
Then, for all $s>0$ and all $\TT\in\T$, it holds
\begin{align*}
 \norm{u}{\E_s(\TT)}<\infty
 \quad\Longleftrightarrow\quad
 \norm{u}{\A_s}<\infty.
\end{align*}
\end{lemma}%

\begin{proof}
We show that $\norm{u}{\E_s(\TT)}<\infty$ if and only if $\norm{u}{\A_s(\TT)}<\infty$. Then, Lemma~\ref{lemma:As} will conclude the proof.

\emph{Step 1.}\quad
Let $\TT\in\T$ and $\widehat\TT_0\in\T$ from Lemma~\ref{lemma:m}. With $C:=(\Cson^m-1)\#\TT_0$, the triangulation $\TT_\star:=\TT\oplus\widehat\TT_0$ satisfies $\#\TT_\star \le \#\TT + \#\widehat\TT_0 - \oooldRevision{\#\TT_0} \le \#\TT + C$ and hence $\TT_\star\in\T_C(\TT)$. This proves $\TT_N(\TT)\neq\emptyset$ for $N\ge C$.

\emph{Step 2.}\quad We prove that $\norm{u}{\A_s(\TT)}<\infty$ implies $\norm{u}{\E_s(\TT)}<\infty$ by showing
\begin{align}\label{eq:Es:1}
 \sup_{N\ge C}\Big( (N+1)^s \min_{\substack{\TT_\star \in \refine(\TT)\\\#\TT_\star-\#\TT\le N}} \inf_{V_\star\in\XX_\star}\big(\norm{u-V_\star}\HH + \osc_\star(V_\star)\big) \Big) \lesssim \norm{u}{\A_s(\TT)}:
\end{align}
For $N\ge C$, choose $\TT_\plus\in\T_N(\TT)$ with $\eta_\plus = \min_{\TT_\star\in\TT_N(\TT)}\eta_\star$. Then,
\begin{align*}
 \min_{\substack{\TT_\star \in \refine(\TT)\\\#\TT_\star-\#\TT\le N}} \inf_{V_\star\in\XX_\star}\big(\norm{u-V_\star}\HH + \osc_\star(V_\star)\big)
\le \norm{u-U_\plus}\HH+\osc_\plus(U_\plus)\simeq\eta_\plus
 = \min_{\TT_\star\in\TT_N(\TT)}\eta_\star.
\end{align*}
This proves~\eqref{eq:Es:1}.

\emph{Step 3.} We prove that $\norm{u}{\E_s(\TT)}<\infty$ implies $\norm{u}{\A_s(\TT)}<\infty$ by showing
\begin{align}\label{eq:Es:2}
 \sup_{N\ge C}\Big( (N+1)^s \min_{\TT_\star \in \T_N(\TT)} \eta_\star\Big) \le (C+1)^s\,\norm{u}{\E_s}:
\end{align}
Let $N\ge0$. Choose $\TT_\plus \in \refine(\TT)$ with $\#\TT_\plus-\#\TT\le N$ 
and $\big(\norm{u-V_\plus}\HH + \osc_\plus(V_\plus)\big) = \inf_{V_\star\in\XX_\star}\big(\norm{u-V_\star}\HH + \osc_\star(V_\star)\big)$. Define $\TT_\circ:=\TT_\plus\oplus\widehat\TT_0$ and note that $\TT_\circ\in\TT_{N+C}(\TT)$. Together with the C\'ea lemma~\eqref{eq:sauterschwab_cea} and our assumptions on the data oscillations, we obtain for all $V_\circ\in\XX_\circ$,
\begin{align*}
 \eta_\circ \simeq \norm{u-U_\circ}\HH+\osc_\circ(U_\circ)
 &\lesssim \norm{u-U_\circ}\HH+\osc_\circ(V_\circ)+\norm{U_\circ-V_\circ}\HH
 \\&
 \lesssim \norm{u-U_\circ}\HH+\osc_\circ(V_\circ)+\norm{u-V_\circ}\HH
 \stackrel{\eqref{eq:sauterschwab_cea}}\lesssim \norm{u-V_\circ}\HH+\osc_\circ(V_\circ).
\end{align*}
This reveals $\eta_\circ\simeq\inf_{V_\circ\in\XX_\circ}\big(\norm{u-V_\circ}\HH+\osc_\circ(V_\circ)\big)$. Together with $\XX_\circ\supseteq\XX_\plus$, we derive
\begin{align*}
 (N+C+1)^s\min_{\TT_\star\in\T_{N+C}(\TT)}\eta_\star
 &\le (N+C+1)^s\,\inf_{V_\circ\in\XX_\circ}\big(\norm{u-V_\circ}\HH+\osc_\circ(V_\circ)\big)
 \\&
 \le\Big(\frac{N+C+1}{N+1}\Big)^s(N+1)^s\inf_{V_\plus\in\XX_\plus}\big(\norm{u-V_\plus}\HH+\osc_\plus(V_\plus)\big)
 \\&
 \le(C+1)^s\,\norm{u}{\E_s(\TT)}.
\end{align*}
This proves~\eqref{eq:Es:2}.
\end{proof}

\begin{remark}
The assumptions of Lemma~\ref{lemma:osc} are satisfied for residual-based error estimators in the frame of FEM with $\XX_\star:=\SSS^p(\TT_\star)\cap H^1_0(\Omega)$; see~\cite{ckns,cn,ffp}.
For each element $T\in\TT_\star\in\T$, let $\mathcal{F}_T$ denote the set of its facets (i.e., edges for $d=2$).
For arbitrarily chosen $q\ge p-1$, the data oscillations 
\begin{subequations}\label{eq:osc:def}
\begin{align}
 \osc_\star(V_\star)^2:=\sum_{T\in\TT_\star}\osc_\star(T,V_\star)^2 
\end{align}
corresponding to the indicators from~\eqref{eq:indicators} read, for all $T\in\TT_\star$,
\begin{align}
\begin{split}
 \osc_\star(T,V_\star)^2 
 = &h_T^2\,\min_{Q\in\PP^q(T)}\,\norm{f+{\rm div}(A\nabla V_\star)-b\cdot\nabla V_\star-cV_\star-Q}{L^2(T)}^2
 \\&
 + h_T\,\sum_{F\in\mathcal{F}_T}\min_{Q\in\PP^q(T)}\norm{[(A\nabla V_\star)\cdot n]-Q}{L^2(F\cap\Omega)}^2.
\end{split}
\end{align}
\end{subequations}
The constant $C_{\rm osc}$ in Lemma~\ref{lemma:osc} then depends on $q$ and $p$. If $A,b,c$ are piecewise polynomial and if $q$ is chosen sufficiently large, the local contributions simplify to
the well-known data oscillations $\osc_\star(T,V_\star)^2 = h_T^2\,\min\limits_{f_T\in\PP^q(T)}\,\norm{f-f_T}{L^2(T)}^2$ \oldRevision{as for} the Laplace problem.
\end{remark}

\subsection{Main \oldRevision{result}}
\label{section:main}
The following theorem is the main result of this work. It states that Algorithm~\ref{algorithm} does not only guarantee \ooldRevision{(linear)} convergence, but also the best possible  algebraic convergence rate for the error estimator. In explicit terms, suppose that $\norm{u}{\A_s}<\infty$ for some $s>0$. By definition~\eqref{eq:approximatoin_class} of the approximation class, there exists a sequence of meshes $\widehat\TT_\ell\in\T=\refine(\TT_0)$ and corresponding error estimators $\widehat\eta_\ell$ such that $\widehat\eta_\ell\lesssim \big( \# \widehat\TT_\ell - \# \TT_0+1\big)^{-s}$ for all $\ell\in\N_0$. Note that these ``optimal'' triangulations are not necessarily successive refinements but in general even totally unrelated. Therefore, the important implication of the following theorem is that indeed the adaptively generated triangulations $\TT_\ell$ yield the same algebraic decay $s>0$ \ooldRevision{if the marking parameter $0<\theta\ll1$ is sufficiently small}. Overall, Algorithm \ref{algorithm} thus guarantees that the error estimator decays asymptotically with any possible algebraic rate $s>0$.

\begin{theorem}\label{theorem:optimal}
Suppose~\eqref{axiom:stability}--\oooldRevision{\eqref{axiom:infty}} \ooldRevision{with $\XX_\infty=\HH$ (which can, for instance, be enforced by the expanded D\"orfler marking strategy from Proposition~\ref{prop:doerfler})}. Employ the notation of Algorithm~\ref{algorithm}.
\ooldRevision{Let $\widehat\gamma_0>0$ be the lower-bound of the $\inf$-$\sup$ constant~\eqref{eq:infsup} for the uniform refinement $\widehat\TT_0$ from Lemma~\ref{lemma:m}. Let $\ell_3,\ell_5\in\N_0$ be the indices from Lemma~\ref{lemma2:convergence} and Lemma~\ref{lemma:m}, respectively. 
Define $\ell_6:=\max\{\ell_3,\ell_5\}$.} Let $0<\theta<\theta_{\rm opt}:=(1+\Cstab^2\Crel^2/\ooldRevision{\widehat\gamma_0^2})^{-1}$. Then, for all $s>0$, there exists a constant $\Copt>0$ such that
	\begin{align}\label{eq:theorem_optimal_convergence}
		\norm{u}{\A_s}<\infty \quad \Longleftrightarrow \quad \forall \oooldRevision{\ell \geq \ooldRevision{\ell_6}} \quad  \eta_\ell \leq 
		\Copt\,\big( \# \TT_\ell - \# \TT_0+1\big)^{-s}.
	\end{align}%
	The constant $\Copt$ depends only on \oooldRevision{$\#\TT_{\ell_6}$}, $\TT_0$, $\theta$, $s$, and validity of~\eqref{axiom:stability}--\oooldRevision{\eqref{axiom:infty}}. 
\end{theorem}

\begin{lemma}[optimality of D\"orfler marking]
\label{lemma:doerfler}
\ooldRevision{Under the assumptions of Theorem~\ref{theorem:optimal} and
 for all $0<\theta<\theta_{\rm opt}$}, there
exists some $0<\kappa_{\rm opt}<1$ such that for all \ooldRevision{$\TT_\star\in\refine(\TT_{\ell_5})$} and all $\TT_\plus\in\refine(\TT_\star)$, 
\oooldRevision{it holds}
\begin{align}
 \eta_\plus \le \kappa_{\rm opt}\,\eta_\star
 \quad\Longrightarrow\quad
 \theta\,\eta_\star^2
 \le \eta_\star(\RR_{\star,\plus})^2,
\end{align}
where $\RR_{\star,\plus}$ is the (enlarged) set of refined elements from~\eqref{axiom:reliability}.
\end{lemma}

\oooldRevision{\begin{proof}
According to \ooldRevision{Lemma~\ref{lemma:m}}, the discrete solutions $U_\star\in\XX_\star$ and $U_\bullet\in\XX_\bullet$ exist, and the discrete reliability \oldRevision{property}~\eqref{axiom:reliability} holds with the uniform constant \ooldRevision{$\Crel/\widehat\gamma_0$}. \oldRevision{Since stability~\eqref{axiom:stability} holds, we can apply \cite[Proposition~4.12]{axioms}, and the statement of the lemma follows.}
\end{proof}}%

\begin{lemma}\label{lemma2:doerfler}
\ooldRevision{Under the assumptions of Theorem~\ref{theorem:optimal},}
there exist constants $\Cdopt,\Cdopttwo >0$ such that for all $\ell \geq \ooldRevision{\ell_5}$, there exists a set $\RR_\ell\subseteq\TT_\ell$ such that the following holds:
For all $s>0$ with \ooldRevision{$\norm{u}{\A_s(\TT_{\ell_5})}<\infty$}, it holds
\begin{align}\label{eq:lemma_R_bound}
 \#\RR_\ell \leq \Cdopt (\Cdopttwo \ooldRevision{\norm{u}{\A_s(\TT_{\ell_5})}})^{1/s} \eta_\ell^{-1/s}
\end{align}
as well as the D\"orfler marking criterion
\begin{align}\label{eq:lemma_implies_doerfler}
 \theta \eta_\ell^2 \leq \eta_\ell(\RR_\ell)^2.
\end{align}
The contant $\Cdopttwo$ depends \ooldRevision{only} on $\theta$, $\ooldRevision{\widehat\gamma_0}$, and~\eqref{axiom:stability}--\eqref{axiom:reliability}, while $\Cdopt$ additionally depends on \ooldRevision{$\#\TT_{\ell_5}$} and $\TT_0$.
\end{lemma} 

\begin{proof}
If $\eta_\ell=0$, the claim~\eqref{eq:lemma_R_bound}--\eqref{eq:lemma_implies_doerfler} is satisfied with $\RR_\ell:=\TT_\ell$. Thus, we suppose $\eta_\ell>0$.

{\em Step~1: Construction of mesh $\TT_\star$ and $\RR_\ell:=\RR_{\ell,\star}$.}\quad
Let $\varepsilon := \Cmon^{-1} \kappa_{\rm opt} \eta_\ell>0$. \ooldRevision{Due to $\ell\ge\ell_5$,} quasi-monotonicity of the estimator \ooldRevision{(Lemma~\ref{lemma:m})} yields $\varepsilon \leq \kappa_{\rm opt} \ooldRevision{\eta_{\ell_5}} < \ooldRevision{\norm{u}{\A_s(\TT_{\ell_5})}} < \infty$.
 Choose the minimal $N \in \N$ such that 
$\ooldRevision{\norm{u}{\A_s(\TT_{\ell_5})}} \leq \varepsilon (N+1)^s$. 
This implies $\eps< \ooldRevision{\norm{u}{\A_s(\TT_{\ell_5})}} \leq \varepsilon (N+1)^s$ and hence $N \geq 1$.
Note that $\ooldRevision{\TT_{\ell_5}\in\T_N(\TT_{\ell_5})}$ and hence $\ooldRevision{\T_N(\TT_{\ell_5})}\neq\emptyset$. Choose $\TT_\varepsilon \in \ooldRevision{\T_N(\TT_{\ell_5})}$ with  $\eta_\varepsilon = \min_{\TT_\star \in \ooldRevision{\T_N(\TT_{\ell_5})}} \eta_{\star}$.
Define $\TT_\star :=\TT_\eps \oplus \TT_\ell$. 
Recall that all $\TT_\bullet \in \refine(\ooldRevision{\TT_{\ell_5}})$ and corresponding spaces $\XX_\bullet \supseteq \ooldRevision{\XX_{\ell_5}}$ provide
unique solutions of the discrete formulation~\eqref{eq:discreteform}.
Therefore, we obtain $\TT_\star \in \ooldRevision{\T_N({\TT_{\ell_5}})}$ with Galerkin solution $U_\star \in \XX_\star$.
Let $\RR_\ell:=\RR_{\ell,\star}$ be the set provided by discrete reliability~\eqref{axiom:reliability}.
 
{\em Step~2: Optimality of D\"orfler marking yields~\eqref{eq:lemma_implies_doerfler}.}\quad
With \oooldRevision{the quasi-monotonicity of the estimator \ooldRevision{(Lemma~\ref{lemma:m})}} and the definition of the approximation class~\eqref{eq:approximatoin_class}, the choice of $N$ yields
\begin{align*}
 \eta_\star \le \Cmon \eta_\varepsilon \stackrel{\eqref{eq:approximatoin_class}}\leq \Cmon (N+1)^{-s} \ooldRevision{\norm{u}{\A_s(\TT_{\ell_5})}} \leq \Cmon \varepsilon = \kappa_{\rm opt} \eta_\ell.
\end{align*}
This implies $\eta_\star \leq \kappa_{\rm opt} \eta_\ell$ and hence 
Lemma~\ref{lemma:doerfler} proves~\eqref{eq:lemma_implies_doerfler}.

{\em Step~3: Verification of~\eqref{eq:lemma_R_bound}.}\quad
The choice $\RR_\ell=\RR_{\ell,\star}$ together with $\TT_\ell,\TT_\eps\in\refine(\ooldRevision{\TT_{\ell_5}})$ yields
\begin{align}\label{eq:proof_dopt_R_bound}
	\# \RR_{\ell} \stackrel{\eqref{axiom:reliability}}{\leq} \Crel \# (\TT_\ell \setminus \TT_\star) 
	\stackrel{\eqref{mesh:sons}}{\leq} \Crel (\# \TT_\star - \# \TT_\ell)
	\stackrel{\eqref{mesh:overlay}}{\leq} \Crel  (\# \TT_\eps - \ooldRevision{\# \TT_{\ell_5}}) \leq \Crel N.
\end{align}
Finally, minimality of $N$ implies
\begin{align*}
	N < \ooldRevision{\norm{u}{\A_s(\TT_{\ell_5})}}^{1/s} \varepsilon^{-1/s} = \Cdoptthree \eta_\ell^{-1/s},
\end{align*}
with $\Cdoptthree := \ooldRevision{\norm{u}{\A_s(\TT_{\ell_5})}}^{1/s} (\Cmon^{-1} \kappa_{\rm opt})^{-1/s} = (\Cmon \kappa_{\rm opt}^{-1} \ooldRevision{\norm{u}{\A_s(\TT_{\ell_5})}})^{1/s}$. Altogether, we thus see
\begin{align*}
\#\RR_\ell \stackrel{\eqref{eq:proof_dopt_R_bound}}\le 
\Crel N < \Crel \Cdoptthree \eta_\ell^{-1/s}.
\end{align*}
This proves~\eqref{eq:lemma_R_bound} with $\Cdopt = \Crel$  
and $\Cdopttwo = \Cmon \kappa_{\rm opt}^{-1}$. 
\end{proof}

\begin{proof}[Proof of Theorem~\ref{theorem:optimal}]
The implication ``$\Longleftarrow$'' in~\eqref{eq:theorem_optimal_convergence} follows by definition of the approximation class (cf.~\cite[Proposition~4.15]{axioms}). We thus focus on the implication ''$\Longrightarrow$'' in~\eqref{eq:theorem_optimal_convergence}.
To this end, suppose that $\norm{u}{\A_s}<\infty$. Lemma~\ref{lemma:As} then implies $\ooldRevision{\norm{u}{\A_s(\TT_{\ell_5})}}<\infty$. For $\ell \geq \ooldRevision{\ell_6=\max\{\ell_3,\ell_5\}}$, let $\MM_\ell$ be the set of marked elements in the $\ell$-th step of Algorithm~\ref{algorithm}. According to Lemma~\ref{lemma2:doerfler},
there exists $\RR_{\ell}\subseteq\TT_\ell$ with~\eqref{eq:lemma_R_bound}--\eqref{eq:lemma_implies_doerfler}. According to the minimality of $\MM_\ell$ (see step~4 in Algorithm~\ref{algorithm}), it follows
\begin{align*}
	\# \MM_\ell \stackrel{\eqref{eq:lemma_implies_doerfler}}\leq \Cmark \#\RR_{\ell} \stackrel{\eqref{eq:lemma_R_bound}}{\leq} \Cmark \Cdopt (\Cdopttwo \ooldRevision{\norm{u}{\A_s(\TT_{\ell_5})}})^{1/s} \eta_\ell^{-1/s}.
\end{align*}
With the mesh-closure estimate \eqref{mesh:mesh-closure} and $\Cmesh\ge1$, we further obtain
\begin{align}\label{eq:optimal_convergence_proof}
	\# \TT_\ell - \oooldRevision{\# \TT_{\ell_6}} +1 \leq \Cmesh \sum_{\oooldRevision{j=\ell_6}}^{\ell} \# \MM_j \leq \Cmesh \Cmark \Cdopt (\Cdopttwo \ooldRevision{\norm{u}{\A_s(\TT_{\ell_5})}})^{1/s} \sum_{\oooldRevision{j=\ell_6}}^{\ell} \eta_j^{-1/s}.
\end{align}
Linear convergence (Theorem ~\ref{theorem:convergence}) implies
\begin{align}
	\eta_\ell \leq \Clin \qlin^{\ell -j} \eta_j \quad \text{for all } \, \oooldRevision{\ell_3} \leq j \leq \ell
\end{align}
and hence
\begin{align*}
	\eta_j^{-1/s} \leq \Clin^{1/s} \qlin^{(\ell -j)/s} \eta_\ell^{-1/s}. 
\end{align*}
Since there holds $0<q:= \qlin^{1/s}<1$, the geometric series applies and yields
\begin{align*}
	\sum_{\oooldRevision{j=\ell_6}}^{\ell} \eta_j^{-1/s} \leq \Clin^{1/s} \eta_\ell^{-1/s} \sum_{j=0}^{\ell} q^{(\ell -j)} 
	\leq \frac{\Clin^{1/s}}{1-\qlin^{1/s}} \, \eta_\ell^{-1/s}.
\end{align*}
Combining this estimate with ~\eqref{eq:optimal_convergence_proof}, we derive 
\begin{align*}
\# \TT_\ell - \oooldRevision{\# \TT_{\ell_6}} +1\leq \frac{\Cmesh \Cmark \Cdopt}{1-\qlin^{1/s}} (\Clin \Cdopttwo \ooldRevision{\norm{u}{\A_s(\TT_{\ell_5})}})^{1/s} \eta_\ell^{-1/s}. 
\end{align*}
Rearranging these terms, we see $\eta_\ell\lesssim(\# \TT_\ell - \oooldRevision{\# \TT_{\ell_6}}+1)^{-s}$. Lemma~\ref{lemma:noe} yields
\begin{align*}
 \# \TT_\ell - \# \TT_{0} +1 
 \stackrel{\eqref{eq:noe}}\le \# \TT_\ell
 \stackrel{\eqref{eq:noe}}\le \oooldRevision{\#\TT_{\ell_6}}\,\big(\# \TT_\ell - \oooldRevision{\#\TT_{\ell_6}} +1\big).
\end{align*}%
This concludes the important implication of~\eqref{eq:theorem_optimal_convergence}. 
\end{proof}

\section{Numerical experiments}%
\label{section:numerics}%
In this section, we present two numerical experiments for
the 2D Helmholtz equation~\eqref{eq:helmholtz} that underpin our theoretical findings. We use the lowest-order FEM with $\XX_\star := \SSS^1(\TT_\star) \cap H^1_0(\Omega)$ and a residual a~posteriori error estimator
(see~\cite{MR1471617} for a first systematic a posteriori error analysis for
the Helmholtz equation and~\cite{MR2201459} for a survey of available error estimation techniques for this problem).
In the experiments, we compare the performance of Algorithm~\ref{algorithm} with respect to
\begin{itemize}
\item different values of $\kappa\in\{1,2,4,8,16\}$,
\item different values of $\theta\in\{0.1,0.2,\dots,0.9\}$,
\item standard D\"orfler marking strategy (with $\Cmark=1$) as well as the expanded D\"orfler marking strategy of Proposition~\ref{prop:doerfler} (with $\Cmark=2$).
\end{itemize}
We consider domains $\Omega \subset \R^2$ with a single re-entrant corner and corresponding interior angle $\alpha>\pi$. Note that elliptic regularity thus predicts a generic convergence order $\OO(N^{-\beta/2})$ for the error on uniform meshes with $N$ elements, where $\beta = \pi/\alpha < 1$. On the other hand, the optimal convergence behavior for lowest-order elements is $\OO(N^{-1/2})$ if the mesh is appropriately refined.

\begin{figure}
  \psfrag{x}{\tiny $x$}
  \psfrag{y}{\tiny $y$}
  \psfrag{nE}[c][c]{\tiny number of elements $N$}
  \psfrag{est:d}{\tiny D\"orfler}
  \psfrag{est:dp}{\tiny expanded D\"orfler}
    \psfrag{est uniform}{\tiny uniform}
  \psfrag{O1}[b][r]{\tiny $\OO\big(N^{-1/2}\big)$}
  \psfrag{O12}[c][l]{\tiny $\OO\big(N^{-\beta/2}\big)$}
  \includegraphics[width=0.53\textwidth]{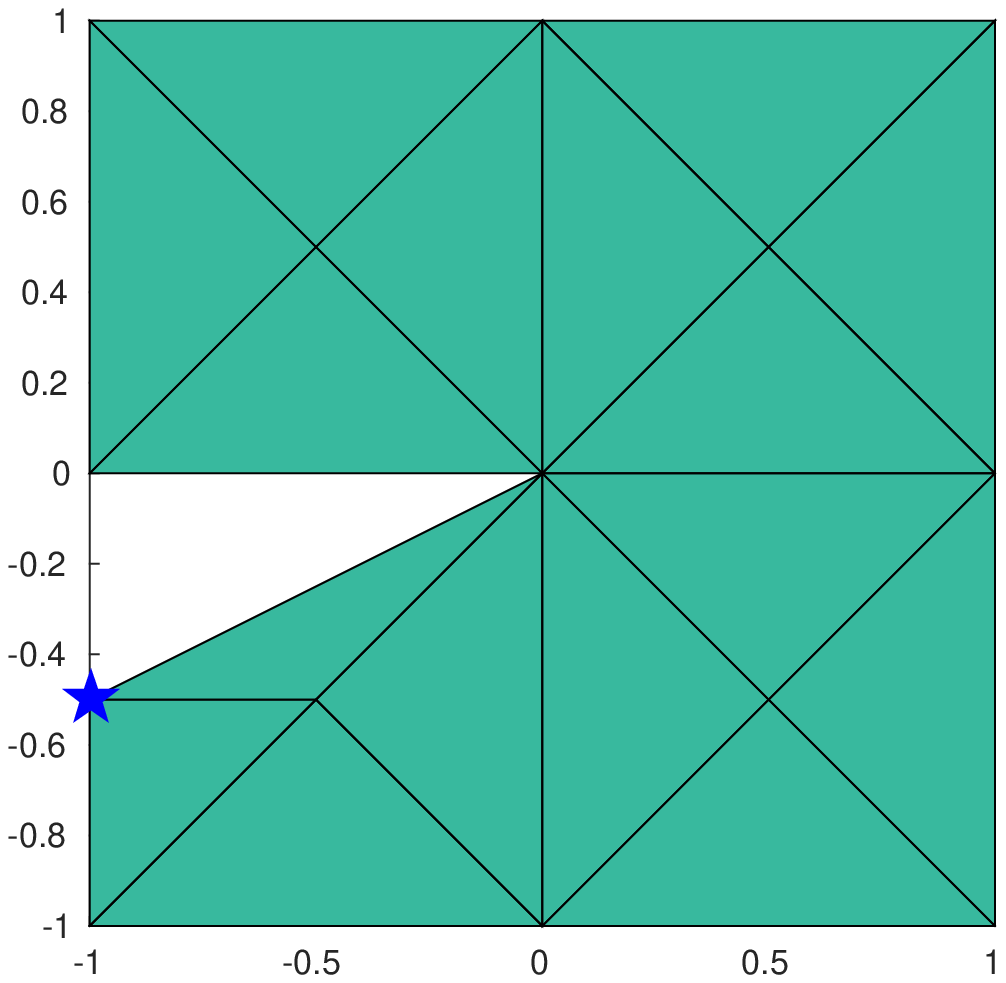}
  \hfill
  \includegraphics[width=0.46\textwidth]{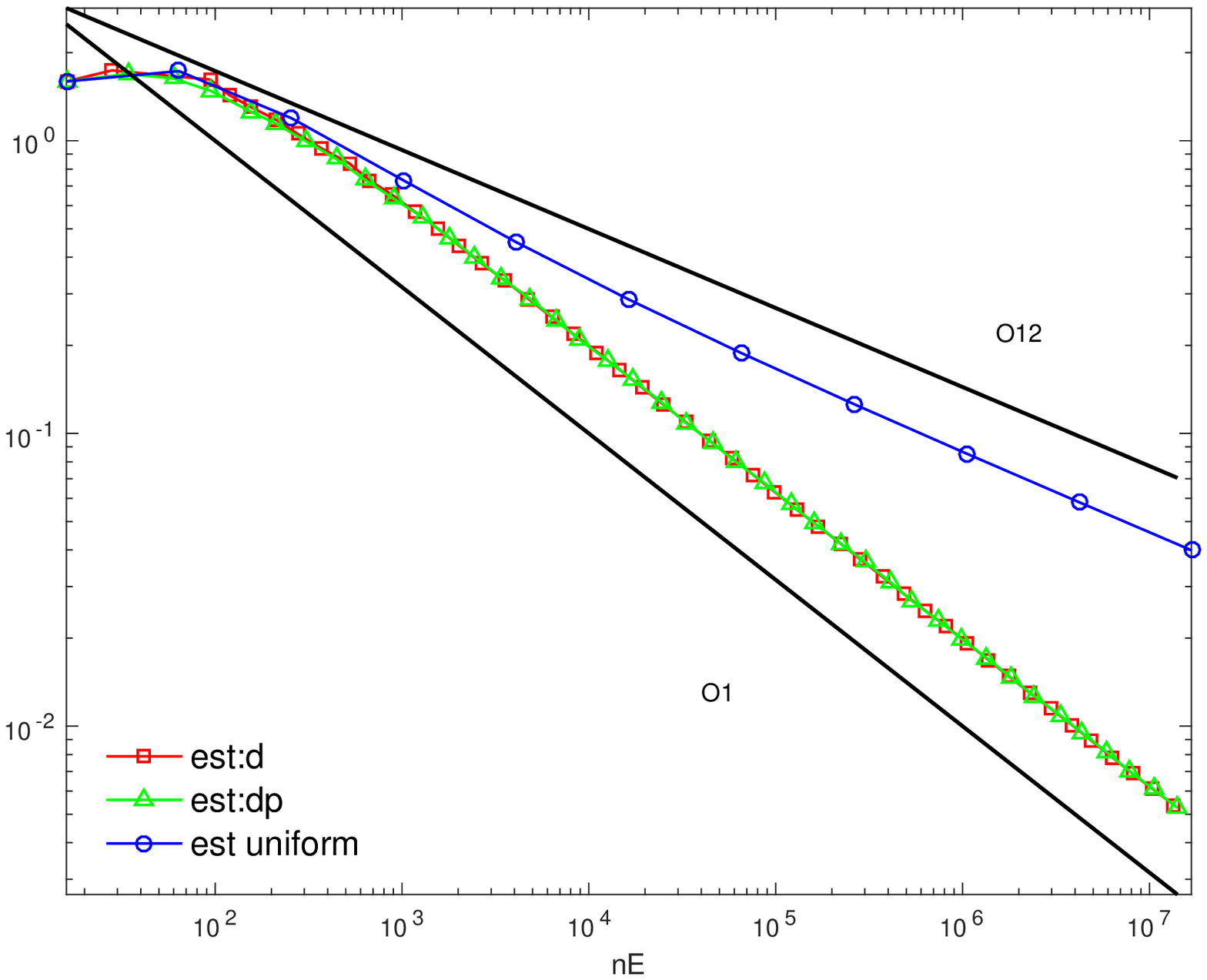}

  \caption{Geometry and initial partition $\TT_0$ in the experiment from Section~\ref{section:ex1} (left), where the blue star indicates the node $(-1,-t) = (-1,-0.5)$. 
  For $\kappa =2$ , we compare the error estimator for uniform vs.\ adaptive mesh-refinement with $\theta=0.2$ (right). Uniform mesh-refinement leads to a suboptimal convergence rate, while Algorithm~\ref{algorithm} with D\"orfler marking and expanded D\"orfler marking recovers the optimal convergence rate.}
  \label{fig:ex1:geometry}
\end{figure}
\begin{figure}
  \centering
  \psfrag{nE}[c][c]{\tiny number of elements $N$}
   \psfrag{kappa =0}{\tiny $\kappa=0$}
  \psfrag{kappa =1}{\tiny  $\kappa=1$}
  \psfrag{kappa =2}{\tiny $\kappa=2$}
  \psfrag{kappa =4}{\tiny $\kappa=4$}
	\psfrag{kappa =8}{\tiny $\kappa=8$}
	\psfrag{uniform: kappa =0}{\tiny  $\kappa=0$}
	\psfrag{uniform: kappa =1}{\tiny $\kappa=1$}
	\psfrag{uniform: kappa =2}{\tiny  $\kappa=2$}
	\psfrag{uniform: kappa =4}{\tiny  $\kappa=4$}
	\psfrag{uniform: kappa =8}{\tiny$\kappa=8$}
  \psfrag{O1}[t][r]{\tiny$\OO\big(N^{-1/2}\big)$}
  \psfrag{O12}[b][r]{\tiny $\OO\big(N^{-\beta/2}\big)$}
  \includegraphics[width=0.46\textwidth]{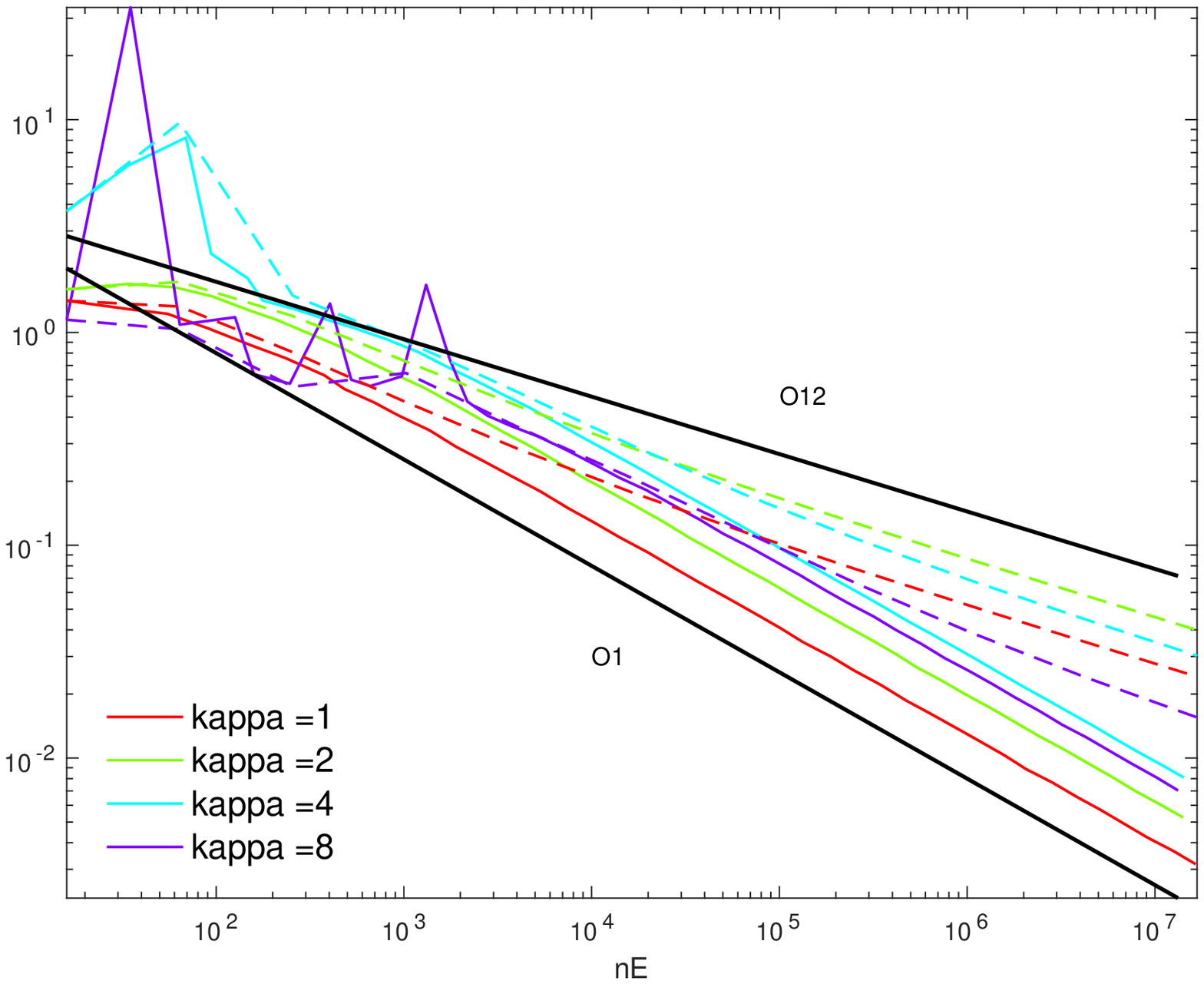}
  \hfill
  \includegraphics[width=0.46\textwidth]{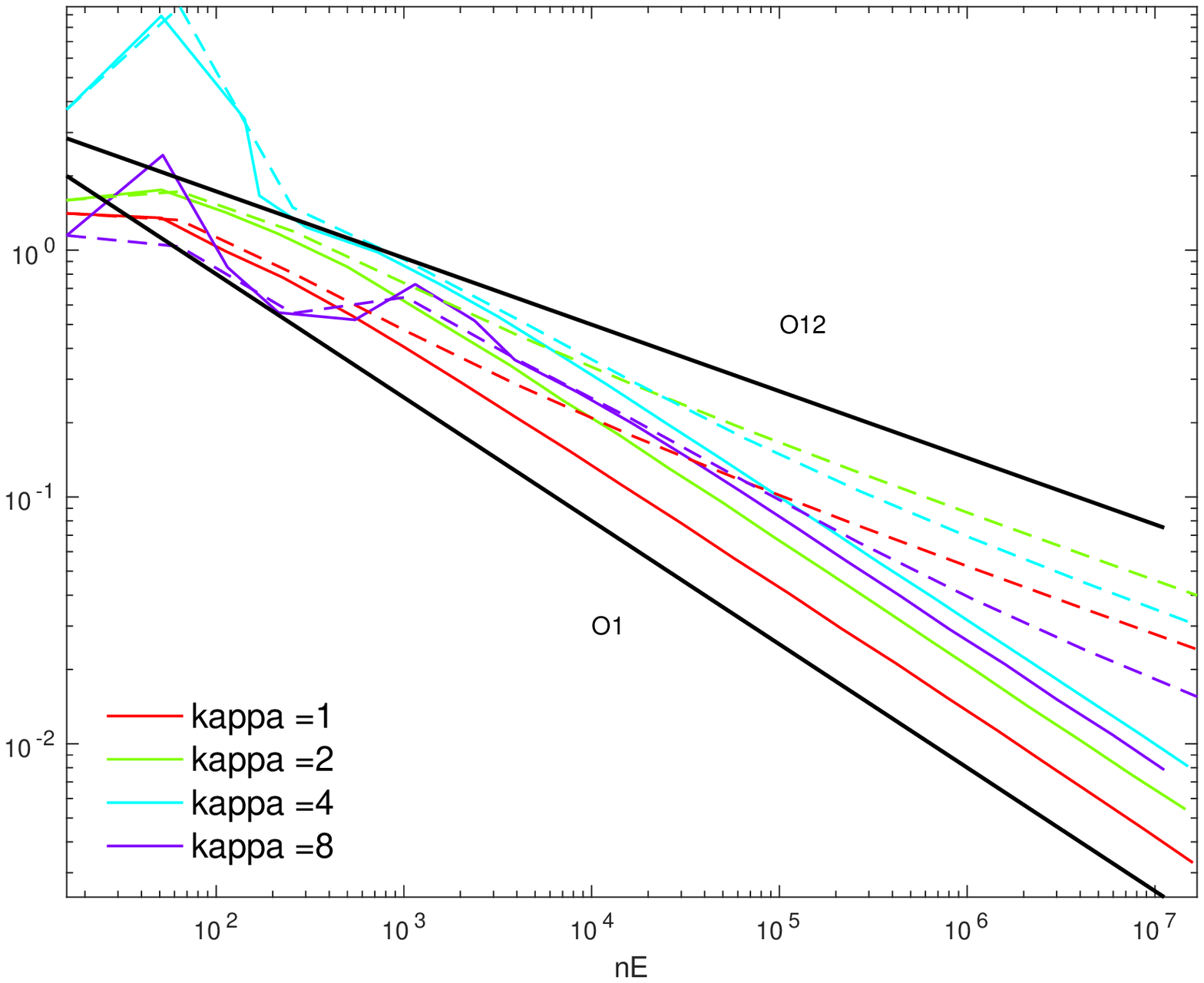}

  \caption{Convergence rates for the error estimator in the experiment from Section~\ref{section:ex1} for different values of $\kappa$ and for marking parameter $\theta=0.2$ (left) and $\theta = 0.5$ (right). Dashed lines mark uniform refinement, while solid lines mark the output of Algorithm~\ref{algorithm} with expanded D\"orfler marking. 
The latter recovers optimal convergence rates, while uniform mesh-refinement does not.\vspace*{2mm}}

  \label{fig:ex1:compare_kappa}
\end{figure}

\begin{figure}
  \centering
  \psfrag{nE}[c][c]{\tiny number of elements $N$}
  \psfrag{theta =0.1}{\tiny $0.1$}
  \psfrag{theta =0.2}{\tiny $0.2$}
  \psfrag{theta =0.3}{\tiny $0.3$}
  \psfrag{theta =0.4}{\tiny $0.4$}
  \psfrag{theta =0.5}{\tiny $0.5$}
  \psfrag{theta =0.6}{\tiny $0.6$}
  \psfrag{theta =0.7}{\tiny $0.7$}
  \psfrag{theta =0.8}{\tiny $0.8$}
  \psfrag{theta =0.9}{\tiny $0.9$} 
  \psfrag{uniform}{\tiny uniform}        
  \psfrag{nE}[c][c]{\tiny number of elements $N$}
  \psfrag{O1}[t][l]{\tiny $\OO\big(N^{-1/2}\big)$}
  \psfrag{O12}[c][t]{\tiny $\OO\big(N^{-\beta/2}\big)$}
  \includegraphics[width=0.46\textwidth]{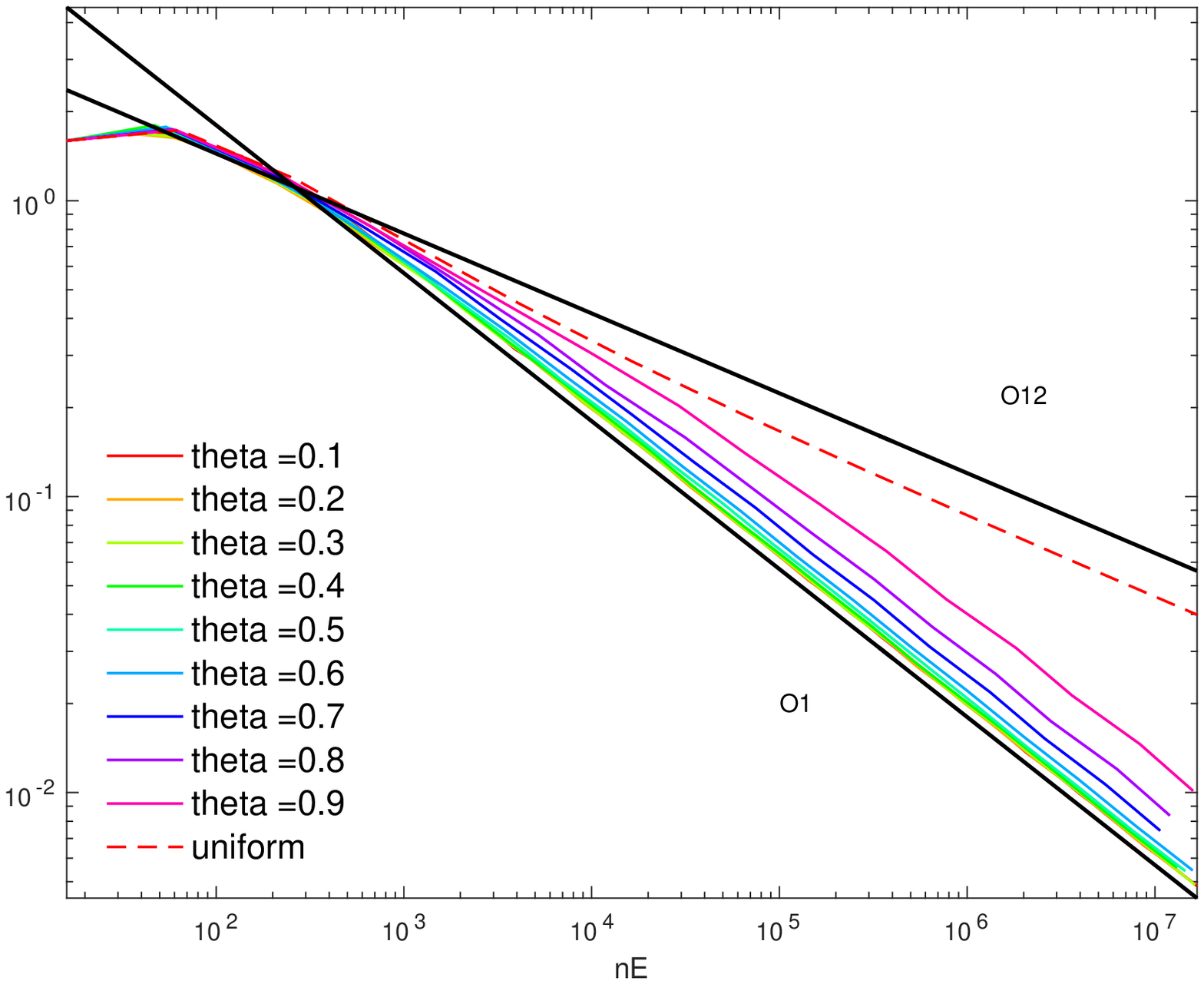}
  \hfill
  \includegraphics[width=0.46\textwidth]{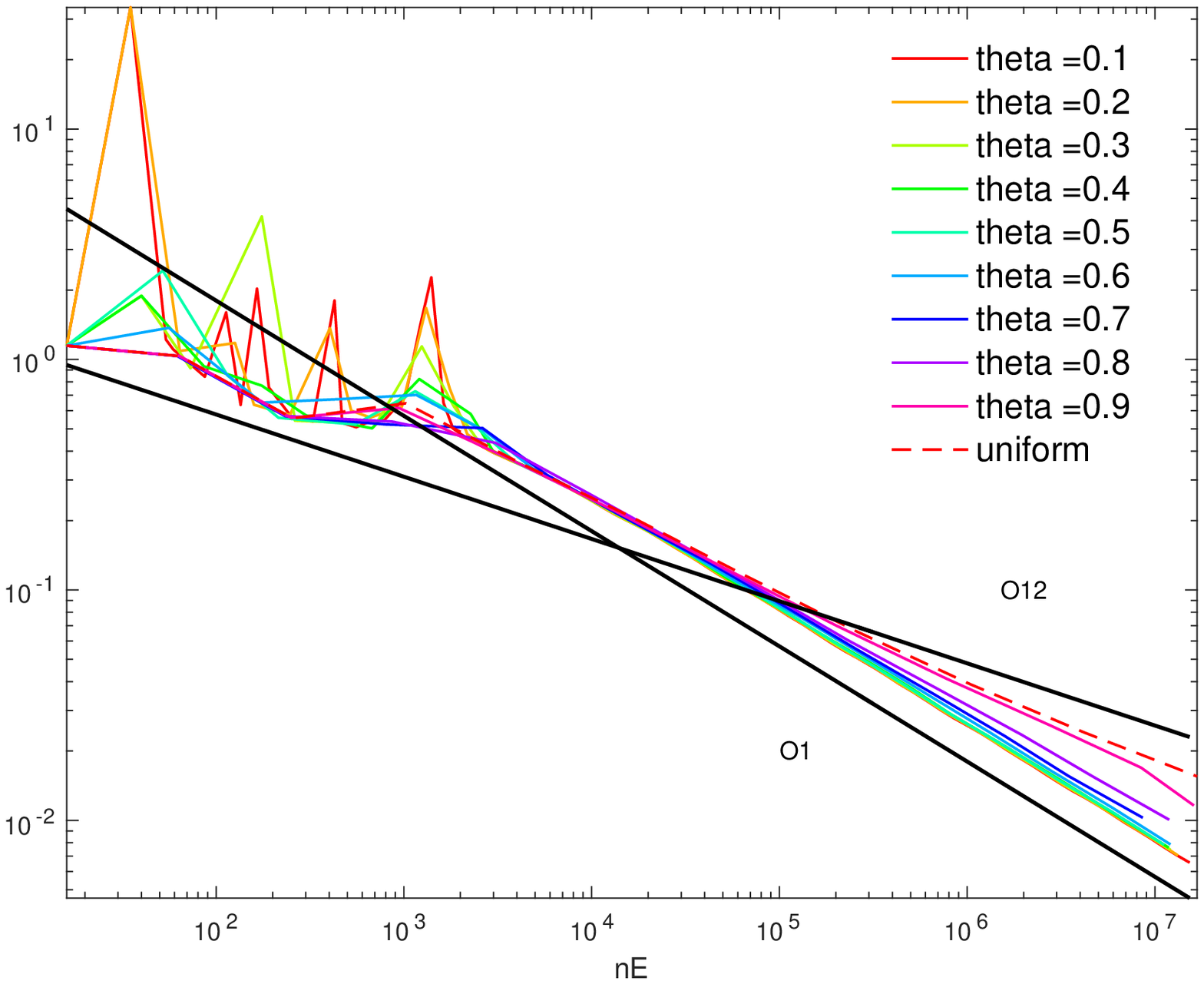}

  \caption{Convergence rates for the error estimator in the experiment from Section~\ref{section:ex1} for uniform and adaptive mesh-refinement with different values of $\theta \in \{0.1, \ldots, 0.9\}$ for $\kappa = 2$ (left) 
  and $\kappa=8$ (right). For all $\theta$, adaptive mesh-refinement leads to optimal convergence behavior, while the preasymptotic behavior increases with $\kappa$.}
   \label{fig:ex1:compare_theta}
\end{figure}

\subsection{Experiment with unknown solution}\label{section:ex1}
We consider the Z-shaped domain $\Omega \subset \R^2$ from Figure~\ref{fig:ex1:geometry}.
The marked node has the coordinates $(-1,-t)=(-1,-0.5)$ and determines the angle $\alpha$ at the re-entrant corner $(0,0)$ which reads $\alpha = 2 \pi- \arcsin\big( t / \sqrt{1+t^2} \big)$, i.e., $\beta\approx0.5398$.
Consider the constant right-hand side $f=1$ in~\eqref{eq:helmholtz} so that the residual error estimator is equivalent to the actual error, i.e., $\eta_\star \simeq \norm{u-U_\star}{H^1(\Omega)}$. For $\kappa=2$, Figure~\ref{fig:ex1:geometry} shows a generically reduced convergence rate for the error estimator on uniform meshes, while Algorithm~\ref{algorithm} with $\theta=0.2$ regains the optimal convergence rate. 
\revision{Empirically, the results generated by employing the standard D\"orfler marking are of no difference to the results generated by employing the expanded D\"orfler marking from Propositon~\ref{prop:doerfler}.}
The same observation is made for other choices of $\theta$ (not displayed), so that we only consider the expanded D\"orfler marking. Figure~\ref{fig:ex1:compare_kappa} compares uniform vs.\ adaptive mesh-refinement for fixed $\theta\in\{0.2,0.5\}$ but various $\kappa\in\{1,2,4,8\}$. As expected, the preasymptotic phase increases with $\kappa$. However, adaptive mesh-refinement results in asymptotically optimal convergence behavior. Figure~\ref{fig:ex1:compare_theta} compares uniform vs.\ adaptive mesh-refinement for fixed $\kappa\in\{2,8\}$ but various $\theta\in\{0.1,\dots,0.9\}$. Although Theorem~\ref{theorem:optimal} predicts optimal convergence rates only for small marking parameters $0<\theta<\theta_{\rm opt}:=(1+\Cstab^2\Crel^2)^{-1}$, we observe that Algorithm~\ref{algorithm} is stable in $\theta$, and any choice of $\theta\le0.9$ leads to the optimal convergence behavior. Finally, we observed that Algorithm~\ref{algorithm} did never enforce uniform mesh-refinement in step~(i), i.e., throughout the resulting discrete linear systems were indefinite but regular.

\begin{figure}
  \centering
  \psfrag{nE}[c][c]{\tiny number of elements $N$}
  \psfrag{est:d}{\tiny est., D\"orfler}
   \psfrag{est:dp}{\tiny est., exp.\ D\"orfler}
    \psfrag{est:uniform}{\tiny est., uniform}
  \psfrag{error:d}{\tiny error, D\"orfler}
   \psfrag{error:dp}{\tiny error, exp.\ D\"orfler}
    \psfrag{error:uniform}{\tiny error, uniform}
  \psfrag{O1}[b][l]{\tiny $\OO(N^{-1/2})$}
  \psfrag{O12}[b][l]{\tiny $\OO(N^{-\beta/2})$}

  \includegraphics[width=0.46\textwidth]{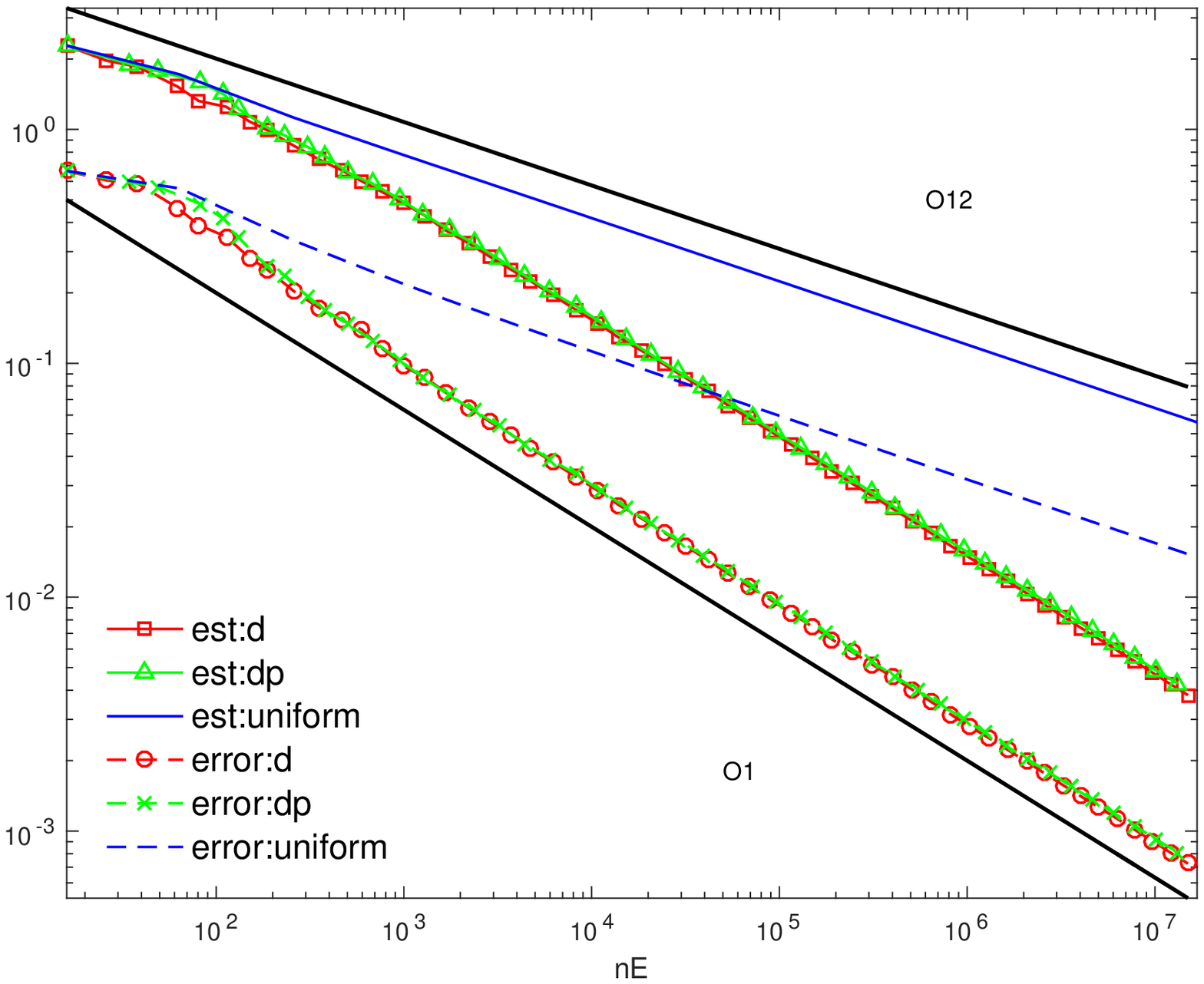}
  \hfill
  \includegraphics[width=0.46\textwidth]{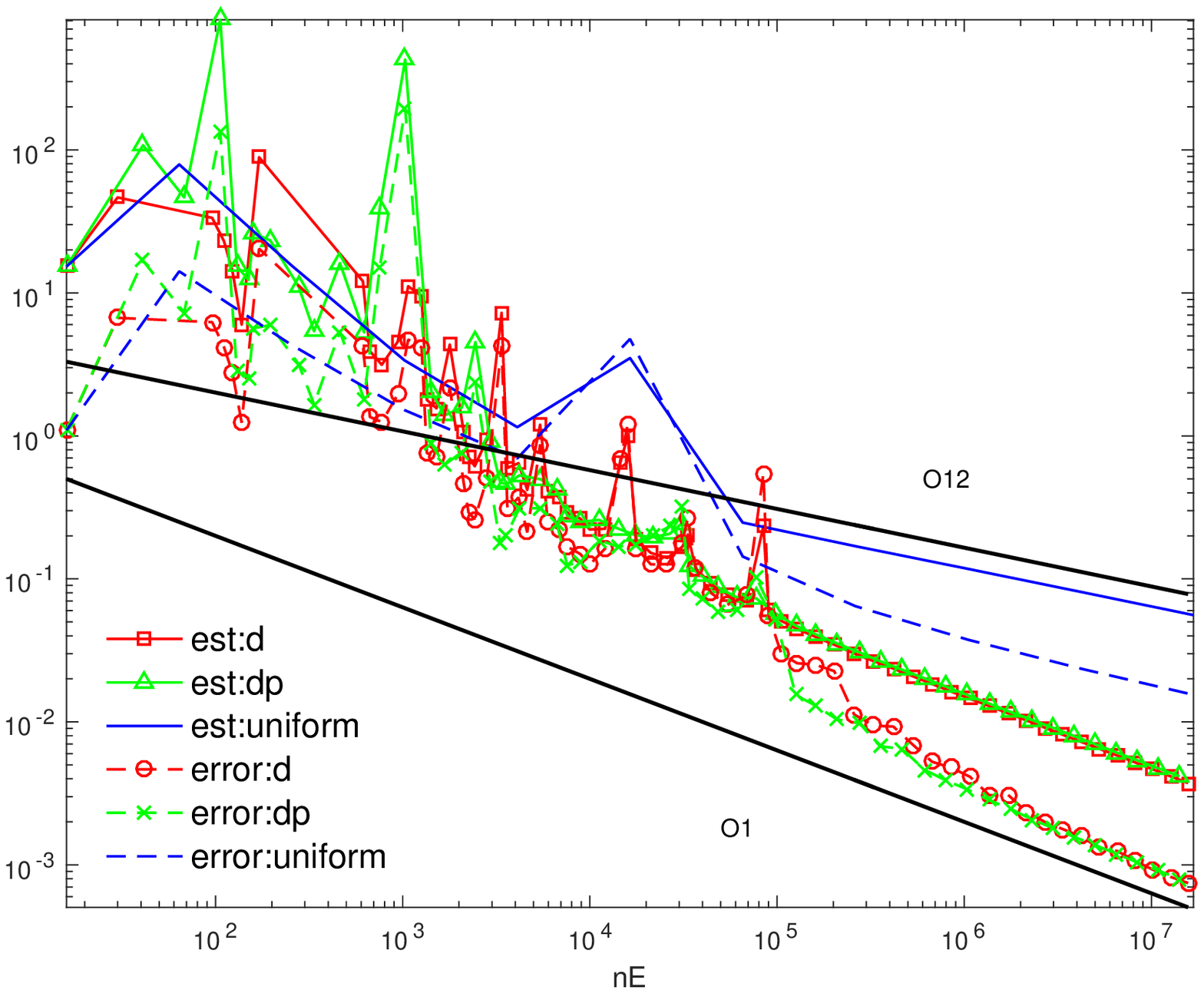}

  \caption{Error and error estimator in the experiment from Section~\ref{section:ex2} for $\kappa = 2$ (left) and $\kappa = 16$ (right) for uniform vs.\ adaptive mesh-refinement with $\theta=0.2$. In particular, we compare 
  D\"orfler marking and expanded D\"orfler marking.}
   \label{fig:ex2:compare_marking}
\end{figure}
\begin{figure}
  \centering
  \psfrag{nE}[c][c]{\tiny number of elements $N$}
  \psfrag{kappa=1}{\tiny  $\kappa=1$}
  \psfrag{kappa=2}{\tiny $\kappa=2$}
  \psfrag{kappa=4}{\tiny  $\kappa=4$}
  	\psfrag{kappa=8}{\tiny  $\kappa=8$}
	\psfrag{kappa=16}{\tiny  $\kappa=16$}
  \psfrag{O1}[t][l]{\tiny $\OO(N^{-1/2})$}
  \psfrag{O12}[c][l]{\tiny $\OO(N^{-\beta/2})$}

  \includegraphics[width=0.46\textwidth]{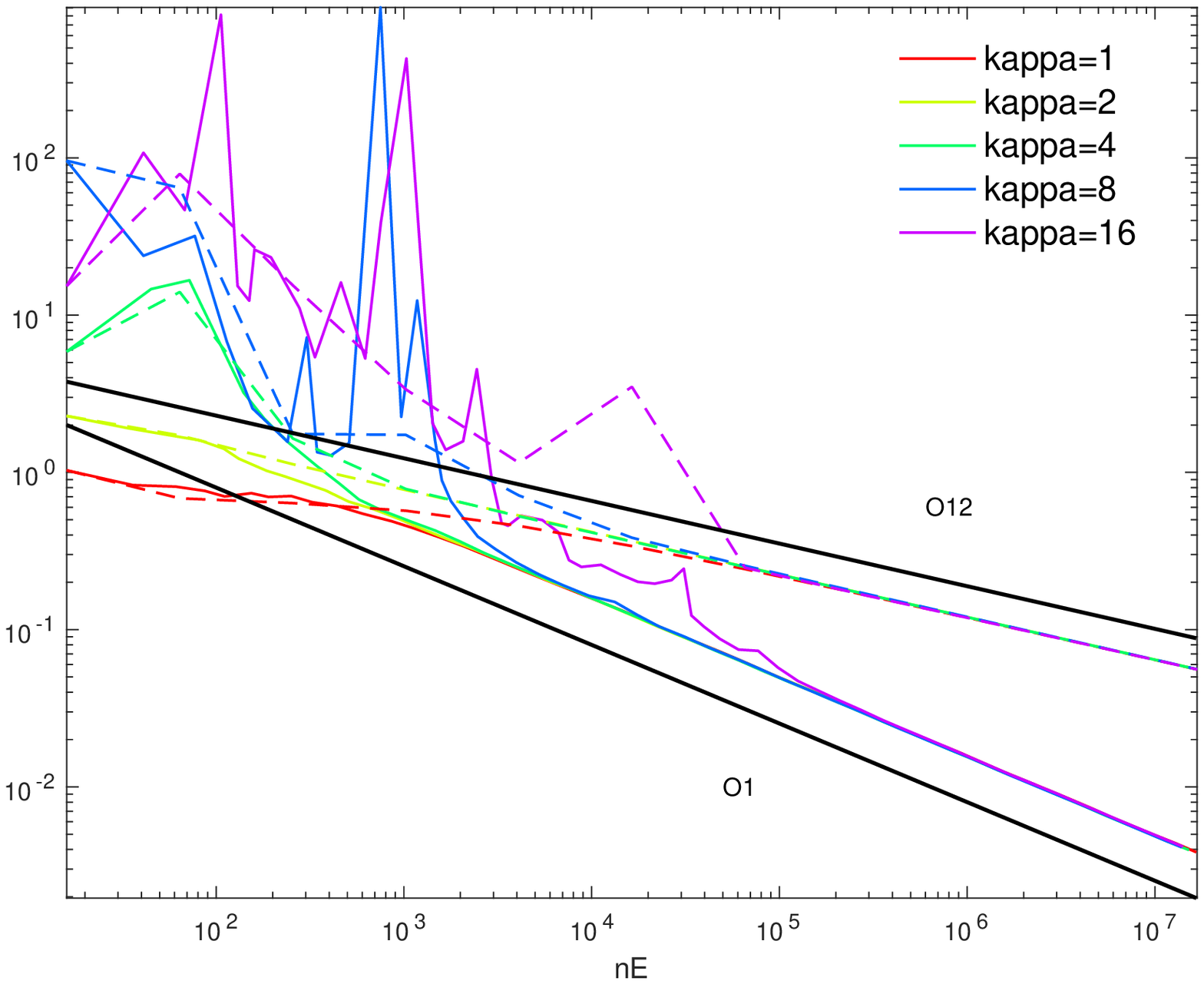}
  \hfill
  \includegraphics[width=0.46\textwidth]{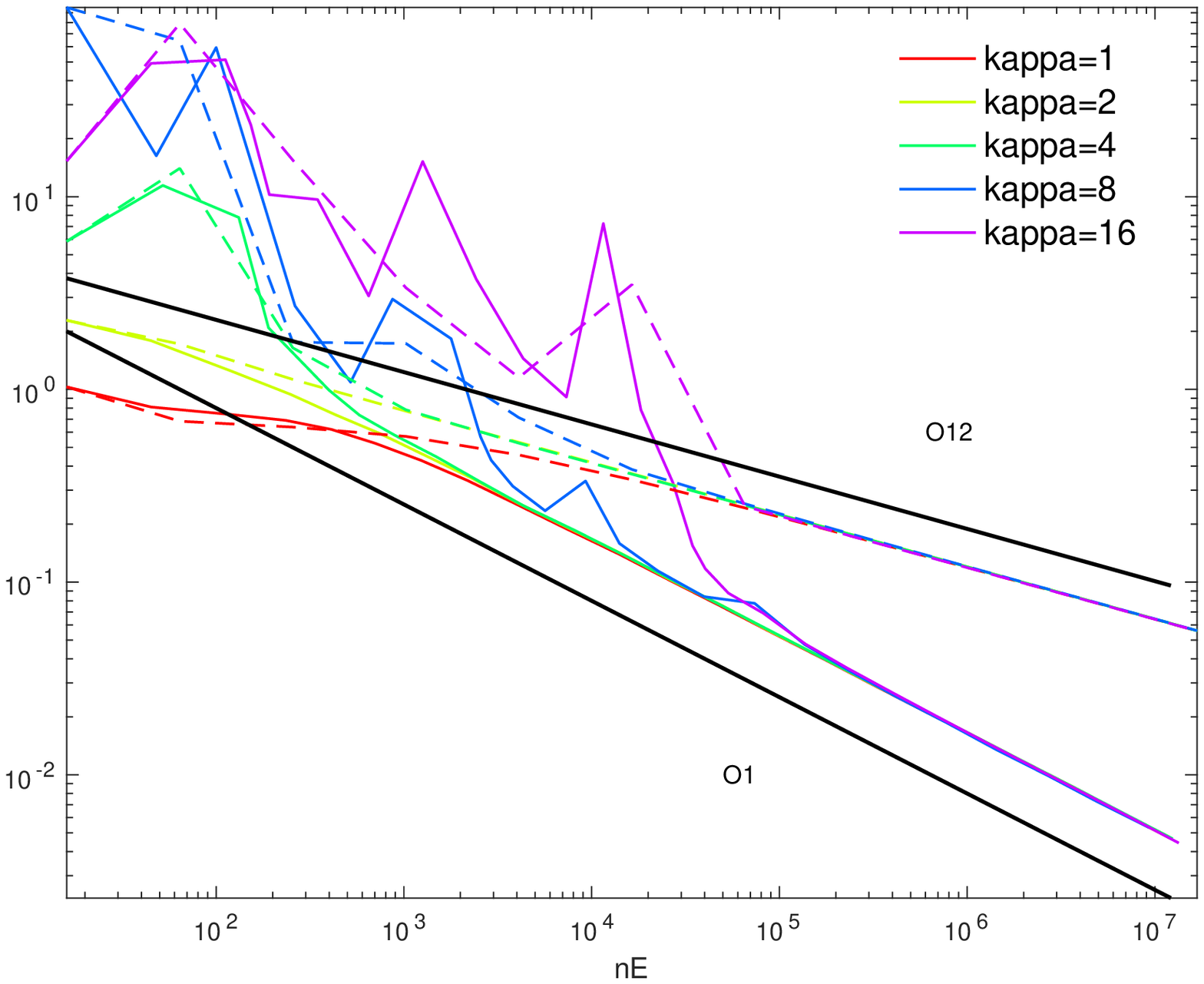}

  \caption{Convergence rates for the error estimator in the experiment from Section~\ref{section:ex2} for different values of $\kappa$ and for marking parameter $\theta=0.2$ (left) and $\theta = 0.5$ (right). Dashed lines mark uniform refinement, while solid lines mark the output of Algorithm~\ref{algorithm} with expanded D\"orfler marking. 
The latter recovers optimal convergence rates, while uniform mesh-refinement does not.\vspace*{2mm}}
  \label{fig:ex2:compare_kappa}
\end{figure}
\begin{figure}
  \centering
  \psfrag{nE}[c][c]{\tiny number of elements $N$}
  \psfrag{theta =0.1}{\tiny $0.1$}
  \psfrag{theta =0.2}{\tiny $0.2$}
  \psfrag{theta =0.3}{\tiny $0.3$}
  \psfrag{theta =0.4}{\tiny $0.4$}
  \psfrag{theta =0.5}{\tiny $0.5$}
  \psfrag{theta =0.6}{\tiny $0.6$}
  \psfrag{theta =0.7}{\tiny $0.7$}
  \psfrag{theta =0.8}{\tiny $0.8$}
  \psfrag{theta =0.9}{\tiny $0.9$} 
  \psfrag{uniform}{\tiny uniform}        
  \psfrag{O1}[t][l]{\tiny $\OO(N^{-1/2})$}
  \psfrag{O12}[l][l]{\tiny $\OO(N^{-\beta/2})$}

  \includegraphics[width=0.46\textwidth]{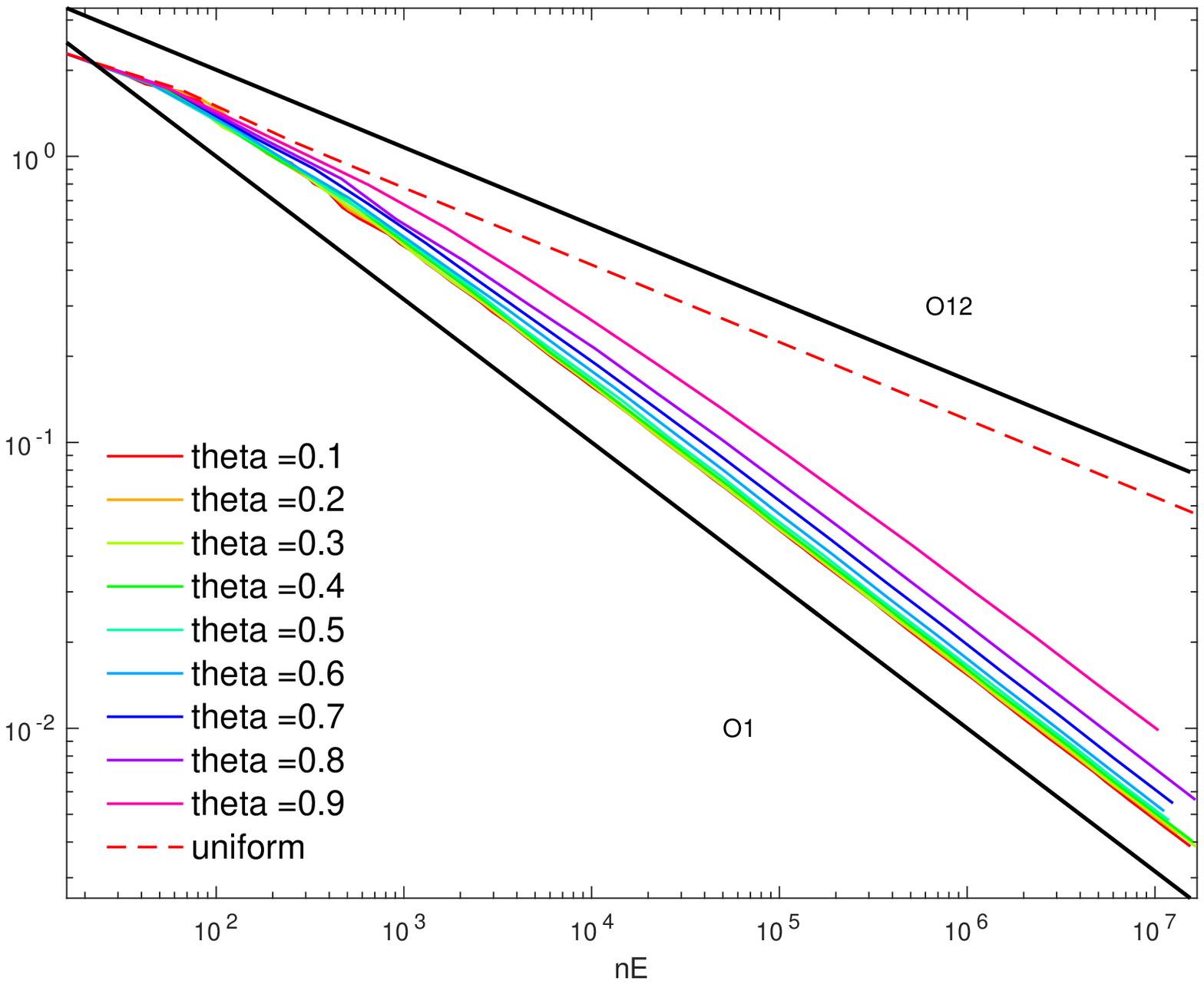}
  \hfill
  \includegraphics[width=0.46\textwidth]{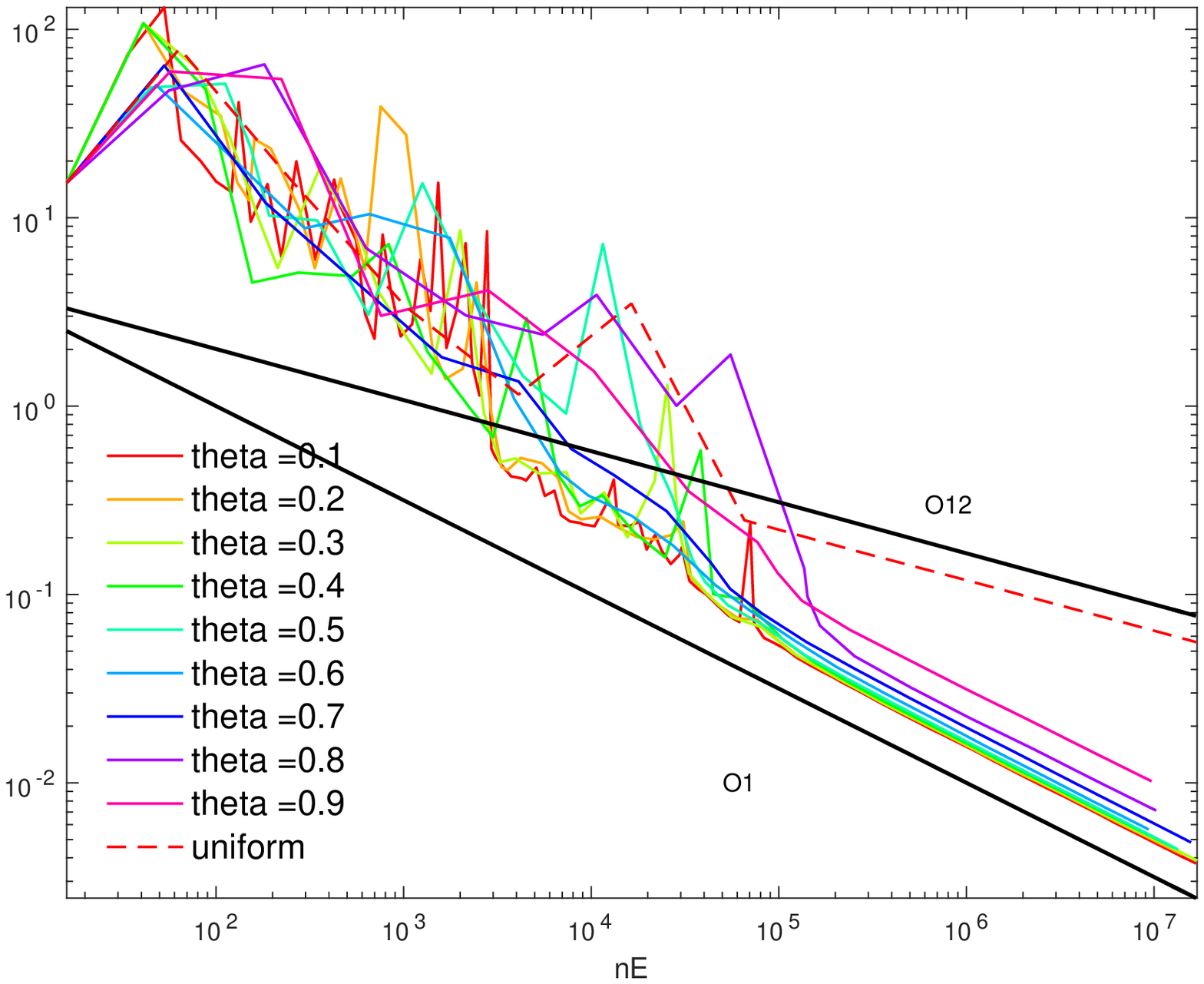}

  \caption{Convergence rates for the error estimator in the experiment from Section~\ref{section:ex2} for uniform and adaptive mesh-refinement with different values of $\theta \in \{0.1, \ldots, 0.9\}$ for $\kappa = 2$ (left) 
  and $\kappa=16$ (right). For all $\theta$, adaptive mesh-refinement leads to optimal convergence behavior, while the preasymptotic behavior increases with $\kappa$.}
  \label{fig:ex2:compare_theta}
\end{figure}
\begin{figure}
\centering
\begin{minipage}[t]{0.32\textwidth}\centering
\includegraphics[width=\textwidth]{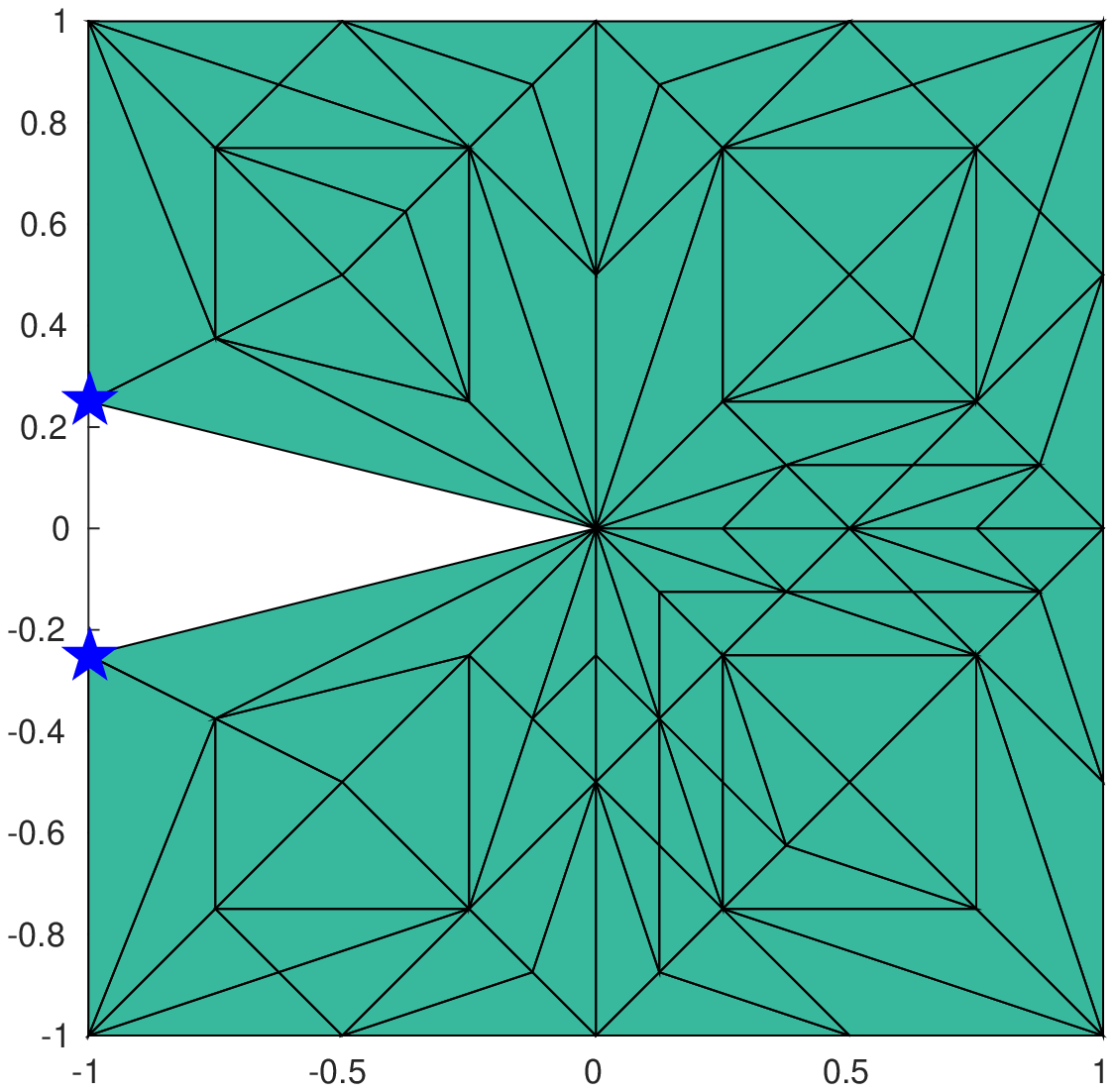}\\
\footnotesize$\#\TT_4 = 110$
\end{minipage}
\begin{minipage}[t]{0.32\textwidth}\centering
\includegraphics[width=\textwidth]{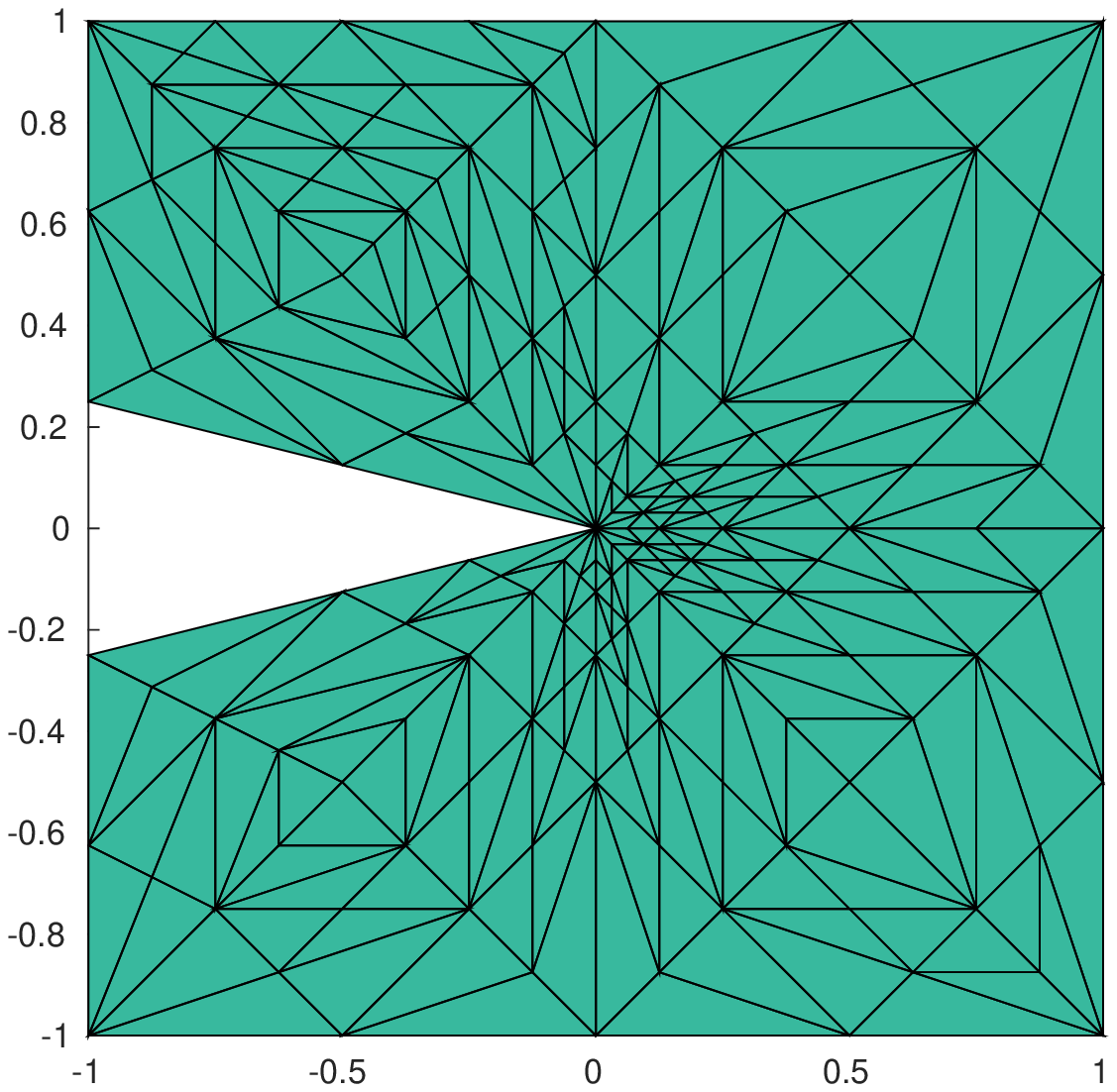}\\
\footnotesize$\#\TT_8 = 309$
\end{minipage}
\begin{minipage}[t]{0.32\textwidth}\centering
\includegraphics[width=\textwidth]{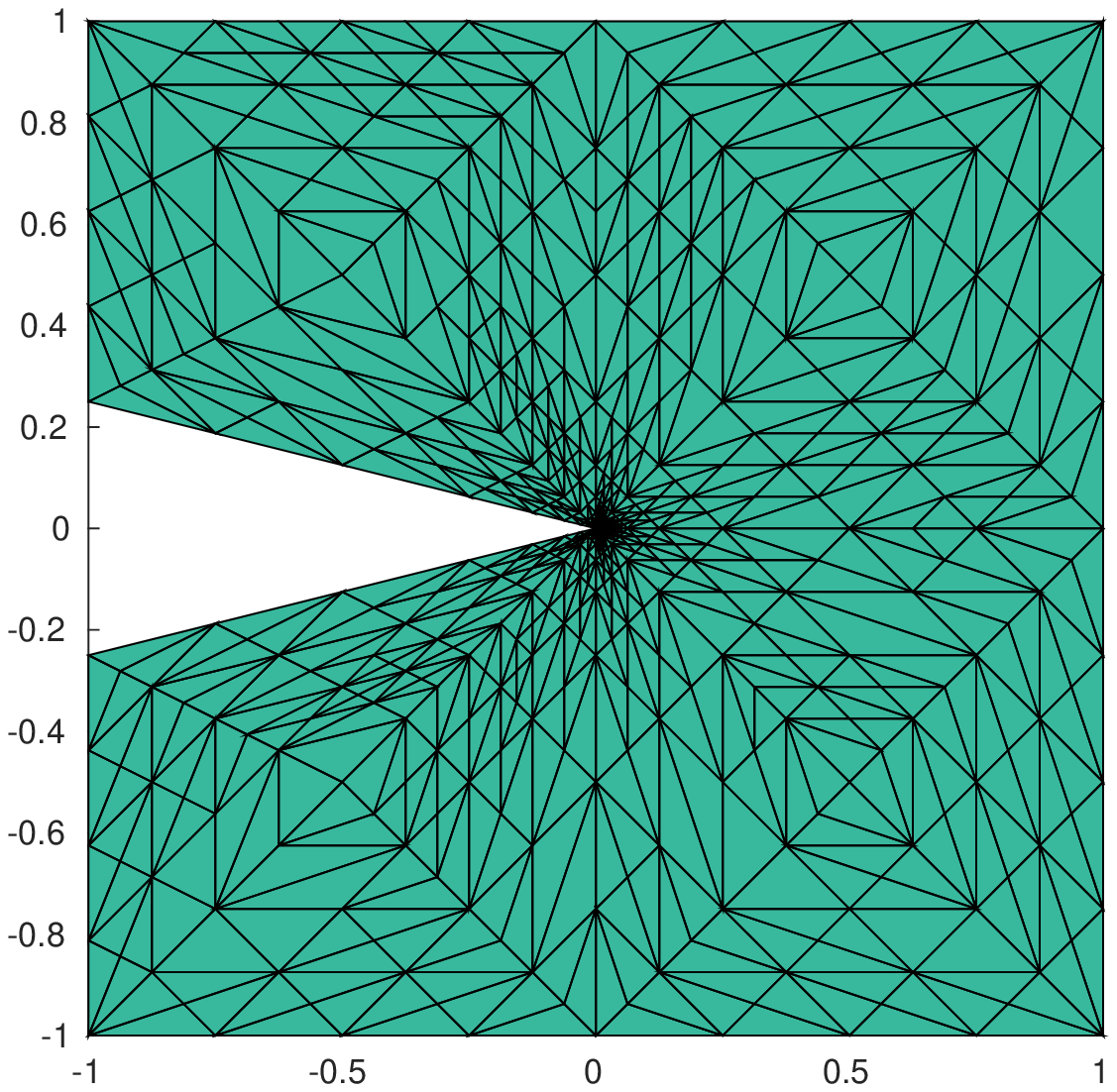}\\
\footnotesize$\#\TT_{12} = 952$
\end{minipage}
\caption{Adaptively generated meshes $\TT_\ell$ in the experiment from Section~\ref{section:ex2} for $\kappa =2$ and $\theta=0.2$.}
\label{fig:ex2:meshes}
\end{figure}
%
\subsection{Experiment with mixed boundary conditions}\label{section:ex2}
We consider a Z-shaped domain with a symmetric opening at the re-entrant corner, see Figure~\ref{fig:ex2:meshes}. The marked nodes read $(-1,\pm t)=(-1, \pm 0.25)$. Analogously to the previous example, we expect a reduced convergence order $\OO(N^{-\beta/2})$ for uniform mesh-refinement with $\beta\approx0.5423$.
We prescribe the exact solution of the Helmholtz equation in polar coordinates $(r,\phi)$ by
\begin{align}\label{eq:ex2:exact_solution}
u(x,y) = r^{\beta} \cos \big( \beta \phi \big)
\end{align}
and define $f:=-\kappa^2u$ in $\Omega$ and $g:=\partial_nu$ on $\Gamma$. Note that $u$ has a generic singularity at the re-entrant corner $(0,0)$ of $\Omega$ and that $u|_{\Gamma_D}=0$ with the Dirichlet boundary $\Gamma_D := {\rm conv}\{(-1,\pm t),(0,0)\}$. Define the Neumann boundary $\Gamma_N := \partial\Omega\backslash\Gamma_D$ and note that $u$ is the unique weak solution 
of the mixed boundary value problem
\begin{align}\label{eq:ex2:problem}
	-\Delta u - \kappa^2u = f \text{ in }\Omega
	\quad\text{subject to}\quad
	u = 0 \text{ on } \Gamma_D \text{ and } \partial_nu = g \text{ on } \Gamma_N.
\end{align}
The weak formulation of this problem can be written in the variational formulation~\eqref{eq:weakform} with $\HH := H^1_D(\Omega) = \set{v\in H^1(\Omega)}{v|_{\Gamma_D}=0\text{ in the sense of traces}}$. Note that assumption~\oooldRevision{\eqref{axiom:infty}} is guaranteed by Proposition~\ref{prop:A:infty} even for standard D\"orfler marking. Moreover, since the exact solution $u$ is given, we can compute the error $\norm{u - U_\star}{H^1(\Omega)}$ besides the corresponding error estimator $\eta_\star$.

The empirical observations are similar to those of Section~\ref{section:ex1}; see Figure~\ref{fig:ex2:compare_marking}--\ref{fig:ex2:compare_theta}. Uniform mesh-refinement leads to suboptimal convergence behavior for both the error and the error estimator. Adaptive mesh-refinement resolves the geometric singularity at the re-entrant corner (see, e.g., Figure~\ref{fig:ex2:meshes}) and recovers the optimal convergence rate. Algorithm~\ref{algorithm} appears to be stable for all $\theta\in\{0.1,\dots,0.9\}$. Different choices of $\kappa\in\{1,2,4,8,16\}$  affect only the preasymptotic phase. Finally, there is no empirical difference between the standard D\"orfler marking and the expanded D\"orfler marking.

\bibliographystyle{alpha}
\bibliography{literature}

\end{document}